\documentclass[12pt]{iopart}

\usepackage{iopams}  
\usepackage{graphicx}
\usepackage[breaklinks=true,colorlinks=true,linkcolor=blue,urlcolor=blue,citecolor=blue]{hyperref}

\usepackage{blindtext}
\usepackage{graphicx}
\usepackage{amsmath}
\usepackage{amssymb}
\usepackage{amsthm}
\usepackage{amsfonts}
\usepackage{float}
\usepackage{color,linegoal}
\usepackage{multicol}
\usepackage{microtype}
\usepackage[shortlabels]{enumitem}
\usepackage[bb=boondox,bbscaled=.9,cal=boondox]{mathalfa}

\newtheorem{theorem}{Theorem}
\theoremstyle{definition}

\newtheorem{lemma}{Lemma}

\numberwithin{equation}{section}
\numberwithin{theorem}{section}
\numberwithin{definition}{section}
\numberwithin{lemma}{section}
\numberwithin{table}{section}
\numberwithin{figure}{section}

\newcommand{\pr}{\partial}
\newcommand{\vph}{\varphi}
\newcommand{\hra}{\hookrightarrow}
\newcommand{\ra}{\rightarrow}

\newcommand{\bR}{\mathbb{R}}
\newcommand{\bC}{\mathbb{C}}
\newcommand{\bE}{\mathbb{E}}
\newcommand{\cq}{\mathcal{q}}
\newcommand{\cV}{\mathcal{V}}
\newcommand{\cH}{\mathcal{H}}
\newcommand{\ca}{\mathcal{a}}
\newcommand{\cA}{\mathcal{A}}
\newcommand{\cT}{\mathcal{T}}
\newcommand{\cB}{\mathcal{B}}
\newcommand{\cC}{\mathcal{C}}
\newcommand{\cx}{\mathcal{x}}

\newcommand{\cy}{\mathcal{y}}

\newcommand{\cL}{\mathcal{L}}
\newcommand{\cU}{\mathcal{U}}
\newcommand{\cS}{\mathcal{S}}
\newcommand{\cI}{\mathcal{I}}
\newcommand{\cP}{\mathcal{P}}
\newcommand{\cJ}{\mathcal{J}}
\newcommand{\cu}{\mathcal{u}}
\newcommand{\cv}{\mathcal{v}}
\newcommand{\cw}{\mathcal{w}}

\newcommand{\cR}{\mathcal{R}}

\DeclareMathOperator*{\argmin}{arg\,min}

\begin{document}

\title[]{Deconvolving the Input to Random Abstract Parabolic Systems; A Population Model-Based Approach to Estimating Blood/Breath Alcohol Concentration from Transdermal Alcohol Biosensor Data$^\dagger$}

\author{Melike Sirlanci}
\address{Modeling and Simulation Laboratory, Department of Mathematics, University of Southern California}
\ead{sirlanci@usc.edu}

\author{Susan E. Luczak}
\address{Department of Psychology, University of Southern California}
\ead{luczak@usc.edu}

\author{Catharine E. Fairbairn}
\address{Department of Psychology, University of Illinois Urbana-Champaign}
\ead{cfairbai@illinois.edu}

\author{Konrad Bresin}
\address{Department of Psychology, University of Illinois Urbana-Champaign}
\ead{bresin2@illinois.edu}

\author{Dahyeon Kang}
\address{Department of Psychology, University of Illinois Urbana-Champaign}
\ead{dkang38@illinois.edu}

\author[cor1]{I. G. Rosen$^*$}
\address{Modeling and Simulation Laboratory, Department of Mathematics, University of Southern California}
\ead{grosen@math.usc.edu}

\begin{abstract}
The distribution of random parameters in, and the input signal to, a distributed parameter model with unbounded input and output operators for the transdermal transport of ethanol are estimated. The model takes the form of a diffusion equation with the input, which is on the boundary of the domain, being the blood or breath alcohol concentration (BAC/BrAC), and the output, also on the boundary, being the transdermal alcohol concentration (TAC). Our approach is based on the reformulation of the underlying dynamical system in such a way that the random parameters are treated as additional spatial variables. When the distribution to be estimated is assumed to be defined in terms of a joint density, estimating the distribution is equivalent to estimating a functional diffusivity in a multi-dimensional diffusion equation. The resulting system is referred to as a population model, and well-established finite dimensional approximation schemes, functional analytic based convergence arguments, optimization techniques, and computational methods can be used to fit it to population data and to analyze the resulting fit. Once the forward population model has been identified or trained based on a sample from the population, the resulting distribution can then be used to deconvolve the BAC/BrAC input signal from the biosensor observed TAC output signal formulated as either a quadratic programming or linear quadratic tracking problem. In addition, our approach allows for the direct computation of corresponding credible bands without simulation. We use our technique to estimate bivariate normal distributions and deconvolve BAC/BrAC from TAC based on data from a population that consists of multiple drinking episodes from a single subject and a population consisting of single drinking episodes from multiple subjects.
\end{abstract}

\vspace{2pc}

\submitto{\IP}
\vspace{1pc}
\noindent{$\dagger$ This research was supported in part by grants from the Alcoholic Beverage Medical Research Foundation and the National Institute of Alcohol Abuse and Alcoholism. (R21AA17711, S.E.L. and I.G.R.), and (R01AA025969,  C.E.F.).  

\noindent$^*$Corresponding author.}

\section{Introduction}\label{intro}

\indent

Alcohol researchers and clinicians have long recognized the potential value of being able to monitor alcohol consumption levels in naturalistic settings. The ability to do so could advance understanding of individual differences in how people choose to drink, respond to alcohol, and behave while drinking, and how patterns and quantities of consumption relate to social versus problematic drinking. There are not adequate methods, however, for passively recording naturalistic drinking in ways that produce accurate quantitative data. Biosensors that measure transdermal alcohol concentration (TAC), the amount of alcohol diffusing through the skin, have been available for several decades and have the potential for passively collecting quantitative levels of alcohol \cite{Swift2000,Swift1993,Swift1992}. Their efficacy is based on the observation that the concentration of ethanol in perspiration, to some extent, correlates with the level of alcohol concentration in the blood \cite{Palmlov1936}. These devices have been used effectively to monitor whether individuals consume any alcohol (e.g., for court-mandated monitoring of abstinence  \cite{Swift2003}). The breath analyzer, which is based on a simple principle from elementary chemistry, Henry's Law \cite{Lab1990}, is reasonably robust and consistent across individuals and ambient conditions thus allowing for the straight forward conversion of breath alcohol concentration (BrAC) to blood alcohol concentration (BAC).  However, because there are variations in the rate at which alcohol diffuses through the skin across individuals and within individuals across environmental conditions, it is challenging to meaningfully interpret TAC levels quantitatively. To wit, TAC levels do not consistently correlate with BrAC or BAC, which are the standard measures of alcohol level intoxication among alcohol researchers and clinicians, and follow a relatively consistent relationship to one another across individuals and environmental conditions. Because raw TAC data does not consistently map directly onto BrAC/BAC across individuals and drinking episodes, alcohol researchers and clinicians have yet to incorporate TAC devices as a fundamental tool in their work. 

Over the past decade, our research team (and others, see, for example, \cite{Dougherty1,Dougherty2,Dougherty3}) has been developing methods to address this conversion problem and produce reliable quantitative estimates of BrAC/BAC (eBrAC/eBAC) from TAC data. To date, we have taken a strictly deterministic approach to converting TAC to eBrAC/eBAC. We created a two-step system that used individual calibration data (i.e., simultaneously-collected breath analyzer BrAC measurements and biosensor TAC measurements) to fit first principles physics/physiological-based models to capture the propagation of alcohol from the blood, through the skin, and its measurement by the TAC sensor (i.e. the forward model). We then deconvolved eBrAC/eBAC from TAC measurements for all other drinking episodes without requiring any additional BrAC measurements. This procedure has produced good results (e.g. \cite{Dai2016,Dum2008,Rosen2014}), and has been used in alcohol related consumption and behavioral studies. Indeed, in \cite{LRWAA} drinking patterns in individuals with and without alcohol metabolizing genetic variants were investigated, and in \cite{Fair} TAC sensors together with our algorithms and software are used to convert the TAC data collected in the field to eBrAC in order to investigate the relationship between social familiarity and alcohol reward in naturalistic drinking settings and compared this to alcohol reward observed in laboratory drinking settings. However, the results of these preliminary studies using this technology also indicate that some of the dynamics of the system are not being captured by the models. In addition, the calibration protocol has limitations, including that the procedure for collecting the individual calibration data is burdensome for both researchers and participants, and that it is not always feasible to conduct (e.g., for clinicians and lay individuals who do not have access to an alcohol administration laboratory or with patients who are trying to abstain). These drawbacks significantly reduce the feasibility of using these devices. Thus, we have been investigating ways to eliminate the need to calibrate the sensor’s data analysis system to each individual, each sensor, and across varying ambient environmental conditions, by better capturing any un-modeled dynamics of the system.  

In a series of recent studies (e.g. \cite{AUTO,SLR17,JMAA}), we have investigated the construction of a population model, which uses our first principles models to describe the dynamics common to the entire population (i.e., all individuals, devices, and environmental conditions) and then to attribute all un-modeled sources of uncertainty observed in individual data (e.g., variations in human physiology, biosensor hardware, environmental conditions) to random effects. We assume that there is a single underlying mathematical framework that describes the system dynamics that are common to all individuals, environments, and devices in the population (e.g., the physics-based model for the transport of alcohol from the blood, through the skin, and measurement by the sensor), but that individuals in the population exhibit variation in the model parameters (e.g., the rate at which the alcohol is transported, evaporates, etc.). We assume that the sensor measures the sum or mean of all of these effects. We refer to the underlying first principles physics based model with random parameters combined with the distribution of these parameters (in the form of parameterized  families of probability measures or, more precisely, joint probability density functions) based on a sample of training data from the population as our population model. In \cite{JMAA} we developed the abstract approximation and convergence theory for fitting or training the population model. In \cite{AUTO} we applied the theory developed in \cite{JMAA} to the alcohol biosensor problem discussed above. In this paper, we are concerned with using the fit population model to deconvolve or an estimate for the input to the model, i.e. the BrAC signal, from the output of the model, i.e. the TAC signal, for an individual who is a member of the population but was not included in the training data set.  In addition to estimating the BrAC based on the TAC, borrowing terminology from Bayesian theory, we want to use the distribution of the random parameters to obtain credible bands for the estimated BrAC. That is, 100$\alpha$ percent credible band is a region surrounding the estimated BrAC signal for which the probability that the true BrAC signal lies in that region is at least $\alpha$.  Our general approach is based on two recent papers on the theory of random abstract parabolic systems (\cite{GAS,SG}) and on an abstract framework for uncertainty quantification and the estimation of probability measures from data for random dynamical systems \cite{BT}.   

An outline of the remainder of the paper is as follows.  In Section 2 we develop a model for the transdermal transport of ethanol and its measurement by the biosensor in the form of an initial-boundary value problem for a diffusion equation with input and output on the boundary.  The input to the system is BrAC and the output is the biosensor measured TAC.  In Section 3 we consider abstract parabolic systems with unbounded input and output operators with random parameters and show how, using the ideas from \cite{GAS}, they can be reformulated as deterministic abstract parabolic systems in appropriately constructed Bochner spaces. Our population model is of this form. In Section 4 we formulate the problem of training or fitting the population model and briefly review our finite dimensional approximation and convergence results from \cite{AUTO} and \cite{JMAA}, as they are fundamental to our approximation and convergence results for the input estimation or deconvolution problem, which is the focus of this paper and in particular, Section 5.  In Section 6, we discuss the matrix representations of the various operators and functionals that are central to our abstract framework.  In Section 7 we present numerical results for two examples involving actual experimental/clinical data.  In one example, we work with a data set consisting of a number of drinking episodes collected from a single subject.   In the second example, we consider data from drinking episodes collected from multiple subjects.  A final eighth section has some discussion and analysis of the results presented in Section 7 along with a number of concluding remarks.

\section{A Mathematical Model of Transdermal Alcohol Transport}\label{alc.trans}

\indent

In this section, we will derive the system of mathematical equations that we use to model the transport of ethanol through the skin.
Let $t$ be the temporal variable and let $\eta$ be the spatial variable. Let $\vph(t,\eta)$ represent the concentration of ethanol in units of moles/cm$^2$ in the epidermal layer of the skin at time $t$ seconds and depth $\eta$ cm. We consider the following system 
\begin{align}
    \frac{\pr\vph}{\pr t}(t,\eta)&=D\frac{\pr^2\vph}{\pr\eta^2}(t,\eta), \ \ 0<\eta<L, \ t>0,\label{eq2.1}\\
    D\frac{\pr\vph}{\pr\eta}(t,0)&=\alpha\vph(t,0), \ \ t>0,\label{eq2.2}\\
    D\frac{\pr\vph}{\pr\eta}(t,L)&=\beta u(t), \ \ t>0,\label{eq2.3}\\
    \vph(0,\eta)&=0, \ \ 0<\eta<L,\label{eq2.4}\\
    y(t)&=\gamma\vph(t,0), \ \ t>0\label{eq2.5},
\end{align}
where the one dimensional diffusion equation \eqref{eq2.1} represents the diffusion of ethanol through the epidermal layer of skin (which does not contain any blood vessels) with thickness $L$ cm and $D$ being the diffusivity coefficient in units of cm$^2$/sec. The boundary condition \eqref{eq2.2} represents the evaporation of ethanol on the skin surface, $\eta=0$, where $\alpha>0$ is the proportionality constant in cm/sec units. The boundary condition \eqref{eq2.3} models the transport of ethanol between the epidermal layer and the dermal layer (which does have blood vessels), where the proportionality constant $\beta>0$ is in units of moles/(cm$\times$sec$\times$BAC/BrAC units), since the input, $u(t)$, is in BAC/BrAC units denoting the ethanol concentration in the blood or exhaled breath. The initial condition \eqref{eq2.4} reflects our assumption that there is no alcohol in the epidermal layer at time $t=0$. Finally, equation \eqref{eq2.5} is called the observation equation and it denotes that the relationship between the TAC sensor reading and the ethanol concentration on the skin surface is linear with a proportionality constant, $\gamma>0$, in units of (TAC units$\times$cm$^2$)/moles, since $y(t)$ is measured in TAC units.

It is possible to convert the system \eqref{eq2.1}-\eqref{eq2.5} into an equivalent one that has only two dimensionless parameters instead of the five parameters, $D$, $L$, $\alpha$, $\beta$, and $\gamma$. These new parameters will be denoted by the vector $q=[q_1,q_2]$ (see \cite{AUTO}) and the system is stated as
\allowdisplaybreaks
\begin{align}
    \frac{\pr\vph}{\pr t}(t,\eta)&=q_1\frac{\pr^2\vph}{\pr\eta^2}(t,\eta), \ \ 0<\eta<1, \  t>0,\label{eq2.6}\\
    q_1\frac{\pr\vph}{\pr\eta}(t,0)&=\vph(t,0), \ \ t>0,\label{eq2.7}\\
    q_1\frac{\pr\vph}{\pr\eta}(t,1)&=q_2u(t), \ \ t>0,\label{eq2.8}\\
    \vph(0,\eta)&=\vph_0, \ \ 0<\eta<1,\label{eq2.9}\\
    y(t)&=\vph(t,0), \ \ t>0,\label{eq2.10}
\end{align}
which is an initial-boundary value problem for a one dimensional diffusion equation with input and output on the boundary where the names of the state, input, and output variables have been kept the same.  Note that we assume that there is no alcohol in the skin at time $t=0$, and consequently $\vph_0=0$.

\section{Random Abstract Parabolic Systems}\label{rand.par.sys}

\indent

In this section we reformulate the system given in \eqref{eq2.6}-\eqref{eq2.10} abstractly in a functional analytic/operator theoretic Gelfand triple setting that will allow us to develop a framework wherein (1) the parameters $q=[q_1,q_2]$ can in a very natural way be considered to be random variables, (2) the distributions of these random parameters can be estimated from data using deterministic techniques, and (3) once the distribution of the random parameters have been estimated, the input signal, $u(t)$, can be estimated based on observations of the output signal, $y(t)$, along with what we refer to as credible bands that quantify the uncertainty in the input that results from the uncertainty in the parameters.  In (2) and (3) above, this includes the development of finite dimensional approximation schemes, rigorous convergence results, and computational/numerical algorithms.       

We assume that we have the Gelfand triple $V\hra H\hra V^*$ with the pivot space $H$ and $V^*$ being the topological dual of $V$. Let $\left<\cdot,\cdot\right>$ denote the $H$ inner product and $|\cdot|$, $||\cdot||$ denote the norms on $H$ and $V$, respectively. Let $Q\subseteq\bR^p$ denote the set of admissible parameters, let $d_Q$ denote a metric on $Q$, and assume that $Q$ is compact with respect to $d_Q$. For $q\in Q$, let $a(q;\cdot,\cdot):V\times V\ra\bC$ be a bounded and coercive (both, uniformly in $q$, for $q$ in the compact set $Q$) sesquilinear form. (The $\lambda_0$-shifted form $a(q;\cdot,\cdot) +\lambda_0|\cdot|^2$ coercive for some $\lambda_0 \in \mathbb{R}$ is fine as well.) We also assume that for each $\varphi,\psi \in V$, the function $q \rightarrow a(q;\varphi,\psi)$ is measurable with respect to all measures $\pi(\theta)$, in some family of measures parameterized by a vector of parameters $\theta$, where $\theta \in \Theta\subset\mathbb{R}^r$ for some positive integer $r$.  This family of measures and its parameterization will be made more precise below (see also \cite{AUTO}). 

Under these conditions, for each $q \in Q$, the sesquilinear form $a(q;\cdot,\cdot)$ defines a bounded linear operator $A(q):V\ra V^*$ by $\left<A(q)\psi_1,\psi_2\right>_{V^*,V}=-a(q;\psi_1,\psi_2)$, $\psi_1,\psi_2\in V$, where $\left<\cdot,\cdot\right>_{V^*,V}$ denotes the duality pairing which is the extension via continuity of the $H$ inner product from $H\times V$ to $V^*\times V$. By appropriately restricting the domain of the operator $A(q)$, it is possible to consider it as an unbounded linear operator on $H$ or $V^*$. Moreover, it can also be shown that the operator $A(q)$ defined above is the infinitesimal generator of an analytic or holomorphic semigroup, $\left\{T(t;q)=e^{A(q)t}, t\geq0\right\}$, of bounded linear operators on $V$, $H$, or $V^*$ (see \cite{BK,BI,TNB}).

The fact that the input $u$ and output $y$ in the system \eqref{eq2.6}-\eqref{eq2.10} are on the boundary of the domain, requires that some care be exercised in the definitions of the input and output operators in our formulation of the problem (see \cite{AUTO}). Let $b(q),c(q)$ denote elements in $V^*$ with the respective maps $q\mapsto<b(q),\psi>_{V^*,V}$ and $q\mapsto<c(q),\psi>_{V^*,V}$ measurable on $Q$ with respect to all measures $\pi(\theta)$ for $\psi \in V$, where $<\cdot,\cdot>_{V^*,V}$ again denotes the duality pairing between $V$ and $V^*$. We assume further that $||b(q)||_{V^*}$, $||c(q)||_{V^*}$ are uniformly bounded for a.e. $q\in Q$.  Then, for $q\in Q$, define the operators $B(q):\mathbb{R}\rightarrow V^*$ by $\left<B(q)u,\varphi\right>_{V^*,V}=\left<b(q),\varphi\right>_{V^*,V}u$ and $C(q):L_2([0,T],V)\rightarrow\mathbb{R}$ by $C(q)\psi=\int_{0}^{T}\left<c(q),\psi(s)\right>_{V^*,V}ds$, for $u\in\mathbb{R}$, $\varphi\in V$, and $\psi \in L_2([0,T],V)$, and consider the input/output state space system given by
\begin{align}
    \dot{x}(t)&=A(q)x(t)+B(q)u(t), \ \ t>0,\label{eq3.1}\\
    x(0)&=x_0\in H,\label{eq3.2}\\
    y(t)&=C(q)x(t),\label{eq3.3}
\end{align}
where $u\in L_2(0,T)$ is the input, $y(t)$ is the output, and $x(t)=\vph(t,\cdot)$ is the state variable. Then by applying the theory for infinite dimensional control systems with unbounded input and output and, in particular, systems described by PDEs with input and output on the boundary of the domain, developed in, for example,  \cite{CS} and \cite{PS}, the mild solution of \eqref{eq3.1}-\eqref{eq3.2} can be written as
\begin{align}
    x(t;q)=T(t;q)x_0+\int_0^t T(t-s;q)B(q)u(s)ds, \ \ t\geq0,\label{eq3.4}
\end{align}
where the state $x$ is in $W(0,T)=\left\{\psi:\psi\in L_2(0,T,V), \dot{\psi}\in L_2(0,T,V^*)\right\}$ and depends continuously on $u$ (\cite{LJL}).

Our model for transdermal alcohol transport described by the system \eqref{eq2.6}-\eqref{eq2.10} can be put in the abstract form given by the system of equations \eqref{eq3.1}-\eqref{eq3.3} by making the following identifications. We let $H=L_2(0,1)$ and $V=H^1(0,1)$ with their standard inner products and the corresponding norms. Consequently we have the continuous and dense embeddings, and hence the Gelfand triple $H^{1}(0,1) \hra L_2(0,1) \hra  H^{-1}(0,1)$. We define the sesquilinear form $a(q;\cdot,\cdot): V \times V \rightarrow \mathbb{R}$ as
\begin{align}
    a(q;\psi_1,\psi_2)=\psi_1(0)\psi_2(0)+q_1\int_0^1\psi_1^{'}(x)\psi_2^{'}(x)dx, \ \ \psi_1,\psi_2\in V,\label{eq3.5}
\end{align}
and the operators $B(q)$ and $C(q)$ as
\begin{align}
    \left<B(q)u,\psi\right>_{V^*,V}=q_2\psi(1)u, \ \ \text{and} \ \ C(q)\psi=\psi(0),\label{eq3.6}
\end{align}
for $\psi\in V$. With these definitions, it is not difficult to show that the boundedness, coercivity, and measurability assumptions on $a(q;\cdot,\cdot)$, $b(q)$, and $c(q)$ are satisfied. A more detailed description of the equaivalence between the model \eqref{eq2.6}-\eqref{eq2.10} and the abstract system \eqref{eq3.1}-\eqref{eq3.3} can be found in \cite{JMAA} and \cite{AUTO}.

In order to include the effects of the un-modeled physiological and environmental factors that were described in the introduction, we now make the assumption that the parameters, $q$, take the form of a random vector which we denote by the $p$-dimensional random vector $\cq$. We assume that $\cq$ has support $Q_0=\prod_{i=1}^p[a^0_i,b^0_i]$ where $-\infty<\bar{\alpha}<a^0_i<b^0_i<\bar{\beta}<\infty$, and has distribution described by a probability measure $\pi_0$ or distribution function $F_0$.  

Now with the parameters $q$ replaced by the random vector $\cq$, the operators $A(\cq)$, $B(\cq)$, $C(\cq)$, and the semigroup $\{T(t;q):t \geq 0\}$ all become random, as does the abstract system given in \eqref{eq3.1}-\eqref{eq3.3} and its solution given in \eqref{eq3.4}.  However, by relying on some recent ideas presented in \cite{GAS} and \cite{SG}, we are able to reformulate the now abstract random system \eqref{eq3.1}-\eqref{eq3.3} as a deterministic system wherein the randomness is embedded in the underlying abstract spaces.  Moreover, this new system continues to be abstract parabolic in that the underlying state transition dynamics are given by a bounded coercive sesquilinear from.  Consequently, a holomorphic semigroup based representation for this new system analogous to the one described above and given by equations \eqref{eq3.1}-\eqref{eq3.3}, \eqref{eq3.5}, and \eqref{eq3.6} can be obtained.  In this way the underlying randomness in the system effectively becomes  invisible, thus allowing us to develop parameter estimation and deconvolution schemes and finite dimensional approximation and convergence theories exactly as we would in the deterministic case.  Indeed, the problem of estimating the distribution of the random vector $\cq$ now becomes essentially the same as estimating a spatial variable dependent diffusivity in a conventional diffusion or heat equation, albeit one with a higher dimensional spatial domain.

To see how this works, define the Bochner spaces $\cV=L_{\pi}^2(Q;V)$ and $\cH=L_{\pi}^2(Q;H)$ which form the Gelfand triple $\cV\hra\cH\hra\cV^*$ with the pivot space $\cH$ and identification of $\cV^*$ with $L_{\pi}^2(Q,V^*)$ (see \cite{GAS}). Then, for $\pi$ a probability measure with corresponding distribution function $F$, we define the $\pi$-averaged sesquilinear form $\ca(\cdot,\cdot):\cV\times\cV\ra\bC$ by
\begin{align}
    \ca(v,w)=\int_Q a(q;v(q),w(q))dF(q),\label{eq3.7}
\end{align}
where $v,w\in\cV$. It is now straightforward to show that $\ca(\cdot,\cdot)$ is also bounded and coercive (see \cite{GAS}). Therefore, in exactly the same way as before, $\ca(\cdot,\cdot)$ defines a bounded linear map $\cA:\cV\ra\cV^*$ which is also the infinitesimal generator of an analytic semigroup, $\left\{\cT(t)=e^{\cA t}:t\geq0\right\}$, of bounded linear operators on $\cV$, $\cH$, or $\cV^*$ depending on its domain. Then, defining the linear operators $\mathcal{B}:\mathbb{R}\rightarrow\mathcal{V}^*$ and $\mathcal{C}:\mathcal{V}\rightarrow\mathbb{R}$ by
\begin{align}
    <\mathcal{B}u,\psi>_{\mathcal{V}^*,\mathcal{V}}&=\int_Q\left<b(q),\psi(q)\right>_{V^*,V}dF(q)u=\mathbb{E}[\left<b(\mathcal{q}),\psi(\mathcal{q})\right>_{V^*,V}|\pi]u,\label{eq3.13}
\end{align}
\begin{align}
    \mathcal{C}\psi=\int_Q\left<c(q),\psi\right>_{V^*,V}dF(q)=\mathbb{E}[\left<c(\mathcal{q}),\psi(\mathcal{q})\right>_{V^*,V}|\pi],\label{eq3.14}
\end{align}
for $u\in\mathbb{R}$ and $\psi\in\mathcal{V}$ (see \cite{JMAA} and \cite{AUTO}), we can now rewrite the system \eqref{eq3.1}-\eqref{eq3.3} as
\begin{align}
    \dot{\cx}(t)&=\cA\cx(t)+\cB u(t), \ \ t>0,\label{eq3.8}\\
    \cx(0)&=\cx_0\in\cH,\label{eq3.9}\\
    \cy(t)&=\cC\cx(t), \ \ t>0.\label{eq3.10}
\end{align}
Then, the mild solution of \eqref{eq3.8}-\eqref{eq3.9} is given by
\begin{align}
    \cx(t)=\cT(t)\cx_0+\int_0^t \cT(t-s)\cB u(s)ds, \ \ t\geq0,\label{eq3.11}
\end{align}
and moreover, from \eqref{eq3.10}, we obtain that
\begin{align}
    \cy(t)=\cC\cT(t)\cx_0+\int_0^t\cC\cT(t-s)\cB u(s)ds, \ \ t\geq0.\label{eq3.12}
\end{align}
At this point, it is important to note (see \cite{GAS} and \cite{SG}) that the solution of system \eqref{eq3.1}-\eqref{eq3.2} is $\pi$-almost everywhere equivalent to the solution of the system \eqref{eq3.8}-\eqref{eq3.9} which is given by \eqref{eq3.11}. Consequently in place of the input output system  \eqref{eq3.1}-\eqref{eq3.3} with random parameters $q$, we consider the system \eqref{eq3.8}-\eqref{eq3.10} with solution given by \eqref{eq3.12}. Since the solution $\cx(\cdot)$ to \eqref{eq3.8}-\eqref{eq3.9} given in \eqref{eq3.11} does not explicitly involve any randomness, we can use the tools of linear semigroup theory to state and prove our approximation and convergence results for both the distribution estimation problem and the deconvolution problem (i.e the problem of estimating the input $u$ based on observations of the output $y$). We refer to the system \eqref{eq3.8}-\eqref{eq3.10}, or equivalently, \eqref{eq3.12}, as the population model.

\section{Estimation of the Distribution of Random Parameters}\label{est.par.dist}

\indent

The estimated distribution of the random parameters, $\cq$, and the corresponding finite dimensional approximation and convergence theory related to it, play a significant role in the deconvolution or input estimation problem which is the central focus of our treatment here.  Consequently, we provide a brief summary of our earlier results in this regard which were reported on in \cite{JMAA} and \cite{AUTO}. We begin by defining the discrete-time version of the system \eqref{eq3.8}-\eqref{eq3.10} and then use it to state the distribution estimation problem.

Let $\tau > 0$ denote a sampling time. We consider zero order hold inputs of the form $u(t)=u_j$, $ t \in [j\tau,(j+1)\tau)$, $j=0,1,...$ . Then, it follows from \eqref{eq3.8}-\eqref{eq3.10} that
\begin{align}
    \cx_{j+1}&=\hat{\cA}\cx_j+\hat{\cB}u_j, \ \ j=0,1,2,...,\label{eq4.1}\\
    \cy_j&=\hat{\cC}\cx_j, \ \ j=0,1,2,...,\label{eq4.2}
\end{align}
with $\cx_0\in\cV$ where $\hat{\cA}=\cT(t)\in\cL(\cV,\cV)$, $\hat{\cB}=\int_0^{\tau}\cT(s)\cB ds\in\cL(\bR,\cV)$, and $\hat{\cC}=\cC\in\cL(\cV,\bR)$. By assumption $\cx_0 = 0$. By the analyticity of the semigroup, $\left\{\cT(t):t\geq0\right\}$, the operators $\hat{\cA}$ and $\hat{\cB}$ are in fact bounded (see \cite{BK,BI,CS,LJL,PS,TNB}).  It follows that $\cx_j \in \cV$, $j=0,1,2, ...$, and moreover, if we make the continuous dependence assumption that  
\begin{align}
   |a(q_1;\psi_1,\psi_2)-a(q_2;\psi_1,\psi_2)| \leq d_Q(q_1,q_2)||\psi_1||||\psi_2||, \ \ \psi_1,\psi_2\in V, q_1,q_2 \in Q,\nonumber
\end{align}
then it is not difficult to show (using the Trotter-Kato theorem from linear semigroup theory, for example, see also,\cite{Kato,TNB}) that $\cx_j \in C(Q;V)$, $j=0,1,2, ...$.  In addition, the coercivity assumption implies that, without loss of generality, we may assume that $\cA:Dom(\cA)\subseteq\cV^*\ra\cV^*$ is invertible with bounded inverse, it follows that $\hat{\cB}=(\hat{\cA}-I)\cA^{-1}\cB\in\cL(\bR,\cV)$. Finally, by the definition of $\hat{\cA}$, we can write
\begin{align}
    \cy_j=\hat{\cC}\cx_j=\bE\left[\cy_j(\cq)|\pi\right]=\int_Q\hat{C}(q)x_j(q)dF(q),\label{eq4.3}
\end{align}
where $\hat{C}$ is the corresponding operator when we write the discrete-time version of the system \eqref{eq3.1}-\eqref{eq3.3} (see \cite{AUTO}). 

Within the framework we have constructed so far, we assume (1) that $\nu$ data sets $(\tilde{u}_i,\tilde{y}_i)_{i=1}^{\nu}=\left(\{\tilde{u}_{i,j}\}_{j=0}^{\mu_i-1},\{\tilde{y}_{i,j}\}_{j=1}^{\mu_i}\right)_{i=1}^{\nu}$ have been collected, and (2) the statistical model given by 
\begin{align}
\tilde{y}_{i,j}=\mathbb{E}[y_{i,j}|\pi_0]+\varepsilon_{i,j}, \ \ j=1,...,\mu_i, \ i=1,...,\nu,\label{eq4.3a}
\end{align}
where in \eqref{eq4.3a}, $\pi_0$ is some unknown probability measure to be estimated, and $\varepsilon_{i,j}, \ j=0,...,\mu_i, \ i=1,...,\nu$, represent measurement noise and are assumed to be independent and identically distributed with mean 0 and common variance $\sigma^2$.  

We let $\vec{a}=[a_i]_{i=1}^p$, $\vec{b}=[b_i]_{i=1}^p$, and $\vec{\theta}\in\Theta$ where $\Theta$ is a parameter set which is a compact subset of $\bR^r$ for some $r$.  Then, in attempting to estimate the distribution $\pi_0$ ($F_0$), we assume that $\cq$ has support of the form $Q=\prod_{i=1}^p[a_i,b_i]$, and that its distribution can be described by a probability measure of the form $\pi(\vec{a},\vec{b},\vec{\theta})$ and corresponding distribution function $q \mapsto F(q;\vec{a},\vec{b},\vec{\theta})$ having the joint probability density function $q \mapsto f(q;\vec{a},\vec{b},\vec{\theta})$. 

Letting $\mathcal{Q}$ be the subset of $\bR^p \times \bR^p$ given by  $\mathcal{Q}= \{(\vec{a}=[a_i]_{i=1}^p,\vec{b}=[b_i]_{i=1}^p): -\infty<\bar{\alpha}<a_i<b_i<\bar{\beta}<\infty\}$, we then state the estimation problem as follows:  We seek
\begin{align}
    (\vec{a}^*,\vec{b}^*,\vec{\theta}^*)=\argmin_{\mathcal{Q}\times\Theta}J(\vec{a},\vec{b},\vec{\theta})=\sum_{i=1}^{\nu}\sum_{j=1}^{\mu_i}\left|\cy_j(\tilde{u}_i,\vec{a},\vec{b},\vec{\theta})-\tilde{y}_{i,j}\right|^2,\label{eq4.4}
\end{align}
where $\left\{\cy_j(\tilde{u}_i,\vec{a},\vec{b},\vec{\theta})\right\}_{j=1}^{\mu_i}$ is given as in \eqref{eq4.3} with $u_j=\tilde{u}_{i,j}$, $j=0,1,...,\mu_i-1$ and $i=1,2,...,\nu$.

Let $\bar{Q}=\prod_{i=1}^p[\bar{\alpha},\bar{\beta}]$.  Then if we require that the maps $(\vec{a},\vec{b},\vec{\theta})\mapsto f(q;\vec{a},\vec{b},\vec{\theta})$ be continuous on $\bR^P\times\bR^P\times\Theta$ for $\pi$-a.e. $q\in\bar{Q}$ and there exist constants $0<\gamma,\delta<\infty$ such that $\gamma<f(q;\vec{a},\vec{b},\vec{\theta})<\delta$, for $\pi$-a.e. $q\in\bar{Q}$, then we can show that the map $(\vec{a},\vec{b},\vec{\theta})\mapsto J(\vec{a},\vec{b},\vec{\theta})$ is continuous. Moreover, using compactness, we can conclude that the estimation problem stated above has a solution, $(\vec{a}^*,\vec{b}^*,\vec{\theta}^*)$. Unfortunately, however, solving this problem in practice requires finite dimensional approximation. In \cite{AUTO} and \cite{JMAA} we have shown how to construct a sequence of approximating optimization problems, each of which has a solution, and for which a subsequence of these solutions converges to a solution to the optimization problem stated in \eqref{eq4.4}, as the level of discretization tends to infinity.  It is this same finite dimensional approximation and convergence framework that is at the heart of our schemes for the deconvolution problem to be discussed in the next section; hence we describe it here. 

We begin by defining the requisite finite dimensional subspaces and corresponding operators defined on these spaces. Let $\bar{\cH}=L_{\pi}^2(\bar{Q};H)$ and $\bar{\cV}=L_{\pi}^2(\bar{Q};V)$ for each $N=1,2,...$. Let $\vec{a}^N=[a_i^N]\in\bR^p$ and $\vec{b}^N=[b_i^N]\in\bR^p$ be such that $-\infty<\bar{\alpha}<a_i^N<b_i^N<\bar{\beta}<\infty$, and $\vec{\theta}^N\in\Theta$, for each $N=1,2,...$ and set $Q^N=\prod_{i=1}^p[a_i^N,b_i^N]$. Then define $\cH^N=L_{\pi}^2(Q^N;H)$, $\cV^N=L_{\pi}^2(Q^N;V)$, and $\cU^N$ to be a finite dimensional subspace of $\cV^N$ for each $N=1,2,...$. Now, define the operator $\cI^N:\bar{\cH}\ra\cH^N$ to be such that $Im(\cI^N)=\cH^N$ and $|\cI^Nx|_{\cH^N}\leq|x|_{\bar{\cH}}$. Then, let $\cP^N:\cH^N\ra\cU^N$ be the orthogonal projection of $\cH^N$ onto $\cU^N$ and define $\cJ^N=\cP^N\circ\cI^N$. Then, define $\cA^N:\cU^N\ra\cU^N$ by
\begin{align}
    \begin{aligned}
        \left<\cA^Nv^N,w^N\right>&=-\ca(v^N,w^N)\\
        &=-\int_{Q^N} a(q;v^N(q),w^N(q))dF^N(q)\\
        &=-\int_{Q^N} a(q;v^N(q),w^N(q))f(q;\vec{a}^N,\vec{b}^N,\vec{\theta}^N)dq,\label{eq4.5}
    \end{aligned}
\end{align}
where $v^N,w^N\in\cU^N$. Note here that it can be shown that $|e^{\cA^Nt}|_{\cH^N}\leq Me^{\omega_0t}$, $t\geq0$ for some constants $M>0$ and $\omega_0$ which are independent of $N$ (see \cite{SLR17,JMAA}). Then, define the operators $\cB^N:\bR\ra\cU^N$ and $\cC^N:\cU^N\ra\bR$ by
\begin{align}
    \left<\cB^Nu,v^N\right>=\int_{Q^N}\left<B(q)u,v^N(q)\right>_{V^*,V}f(q;\vec{a}^N,\vec{b}^N,\vec{\theta}^N)dq,\label{eq4.6}
\end{align}
and
\begin{align}
    \cC^Nv^N=\int_{Q^N}C(q)v^N(q)f(q;\vec{a}^N,\vec{b}^N,\vec{\theta}^N)dq,\label{eq4.7}
\end{align}
where $v^N\in\cU^N$ and $u\in\bR$. Then, assuming that $\nu$ data sets $(\tilde{u}_i,\tilde{y}_i)_{i=1}^{\nu}$ are given, the finite dimensional approximating problems are stated as
\begin{align}
    (\vec{a}^{N*},\vec{b}^{N*},\vec{\theta}^{N*})=\argmin_{\mathcal{Q}\times\Theta}J^N(\vec{a},\vec{b},\vec{\theta})=\sum_{i=1}^{\nu}\sum_{j=1}^{\mu_i}\left|\cy_j^N(\tilde{u}_i,\vec{a},\vec{b},\vec{\theta})-\tilde{y}_{i,j}\right|^2,\label{eq4.8}
\end{align}
where $\{\cy_j^N(\tilde{u}_i,\vec{a},\vec{b},\vec{\theta})\}_{j=1}^{\mu_i}$ is given by
\begin{align}
    \cx_{i,j+1}^N&=\hat{\cA}^N\cx_{i,j}^N+\hat{\cB}^N\tilde{u}_{i,j}, \ \ j=0,1,...,\mu_i-1,\label{eq4.9}\\
    \cy_{i,j}^N&=\hat{\cC}^N\cx_{i,j}, \ \ j=1,2,...,\mu_i,\label{eq4.10}
\end{align}
for each $i = 1,2,...,\nu$, with $\cx_{i,0}^N=0\in\cU^N$, $\hat{\cA}^N=e^{\cA^N\tau}\in\cL(\cU^N,\cU^N)$, $\hat{\cB}^N=\int_0^{\tau}e^{\cA^Ns}\cB^Nds\in\cL(\bR,\cU^N)$, and $\hat{\cC}^N=\cC^N\in\cL(\cU^N,\bR)$. In addition, we require the following assumption. For any $v^N\in\cV^N$, we have
\begin{align}
    \left|\cP^Nv^N-v^N\right|_{\cV^N}\ra0 \ \ \text{as} \ \ N\ra\infty.\label{eq4.11}
\end{align}
Then, under these assumptions, we are able to prove that
\begin{align}
    \left|\cJ^NR_{\lambda}(\cA)x-R_{\lambda}(\cA^N)\cJ^Nx\right|_{\cH^N}\ra0 \ \ \text{as} \ \ N\ra\infty,\label{eq4.12}
\end{align}
for some $\lambda\geq\omega_0$ and every $x\in\bar{\cH}$ where $R_{\lambda}(\cA)$ and $R_{\lambda}(\cA^N)$ denote the resolvent operators of $\cA$ and $\cA^N$ at $\lambda$ (\cite{JMAA}). Using this resolvent convergence result and a version of the Trotter-Kato theorem that allows for the state spaces to depend on the parameters (\cite{BBC,SLR17,JMAA}), we obtain
\begin{align}
    \left|\cJ^N\cT(t)x-e^{\cA^Nt}\cJ^Nx\right|_{\cH^N}\ra0 \ \ \text{as} \ \ N\ra\infty,\label{eq4.13}
\end{align}
for every $x\in\bar{\cH}$ uniformly on compact t-intervals of $[0,\infty)$. Then, these results lead us to show that there exists a subsequence $\left\{(\vec{a}^{N_j*},\vec{b}^{N_j*},\vec{\theta}^{N_j*})\right\}_{j=1}^{\infty}$ that converges to $(\vec{a}^*,\vec{b}^*,\vec{\theta}^*)$ as $j\ra\infty$ where $(\vec{a}^*,\vec{b}^*,\vec{\theta}^*)$ is a solution of the estimation problem \eqref{eq4.4}.

In implementing this scheme, we used linear splines for the $\eta$ discretization and characteristic functions for the $q_1$ and $q_2$ discretizations. Matrix representations of the operators $\hat{\cA}^N$, $\hat{\cB}^N$, and $\hat{\cC}^N$ could then be obtained by using the Galerkin formulation of the operators appearing in the system \eqref{eq4.9}-\eqref{eq4.10} (see Section 6 below for further details). In solving the finite dimensional optimization problems \eqref{eq4.8}, the requisite gradients of the cost functional $J^N$ were computed via the adjoint method (see \cite{Levi2010}). For each $i=1,...,\nu$, set $v_{i,j}^N=[2(\hat{\mathcal{C}}^N\mathcal{x}_{i,j}^N-\tilde{y}_{i,j}),0,...,0]^T\in\mathbb{R}^{K^N}$, $j=0,...,\mu_i$, where $K^N$ is the dimension of $\mathcal{U}^N$ (we note that $\tilde{y}_{i,0}=0$ since we have assumed that $\cx_{i,0} = 0$). The adjoint system is
\begin{align}
    z_{i,j-1}^N=[\hat{\mathcal{A}}^N]^Tz_{i,j}^N+v_{i,j-1}^N.\label{eq4.14}
\end{align}
The gradient of $J^N$ at $\rho = (\vec{a},\vec{b},\vec{\theta})$ is now given by
%
\begin{align}
    \begin{aligned}
        \vec{\nabla}J^N(\rho)&=\sum_{i=1}^\nu\sum_{j=1}^{\mu_i}[z_{i,j}^N]^T\biggl(\frac{\partial\hat{\mathcal{A}}^N}{\partial\rho}\mathcal{x}_{i,j-1}^N\\
        &\quad-(\mathcal{A}^N)^{-1}\biggl(\frac{\partial\mathcal{A}^N}{\partial\rho}(\mathcal{A}^N)^{-1}(\hat{\mathcal{A}}^N-I)\frac{\partial\hat{\mathcal{B}}^N}{\partial\rho}\tilde{u}_{i,j-1}\\
        &\quad-\frac{\partial\hat{\mathcal{A}}^N}{\partial\rho}\hat{\mathcal{B}}^N\tilde{u}_{i,j-1}-(\hat{\mathcal{A}}^N-I)\frac{\partial\hat{\mathcal{B}}^N}{\partial\rho}\tilde{u}_{i,j-1}\biggr)\biggr)\\
        &\quad+\sum_{i=1}^\nu\sum_{j=0}^{\mu_i}\bigl(\mathcal{y}_j^N-\tilde{y}_{i,j}\bigr)^T\frac{\partial\hat{\mathcal{C}}^N}{\partial\rho}\mathcal{x}_{i,j}^N.\label{eq4.15}
    \end{aligned}
\end{align}
Equations \eqref{eq4.14} and \eqref{eq4.15} require the tensor $\frac{\partial\hat{\mathcal{A}}^N}{\partial\rho}$. For $t\geq0$, with $\Phi^N(t)=e^{\mathcal{A}^N(t)}$, it follows that  
\begin{align}
    \begin{aligned}
        \dot{\Phi}^N(t)=\mathcal{A}^N\Phi^N(t), \ \ \Phi^N(0)=I.\label{eq4.16}
    \end{aligned}
\end{align}
Then, if $\Psi^N(t)=\frac{\partial\Phi^N(t)}{\partial\rho}$, the sensitivity equations can be obtained by differentiating \eqref{eq4.16} with respect to $\rho$, and interchanging the order of differentiation,
\begin{align}
    \begin{aligned}
        \dot{\Psi}^N(t)=\mathcal{A}^N\Psi^N(t)+\frac{\partial\mathcal{A}^N}{\partial\rho}\Phi^N(t), \ \ \Psi^N(0)=0.\label{eq4.17}
    \end{aligned}
\end{align}
Equations \eqref{eq4.16} and \eqref{eq4.17} then yield
\begin{align}
\begin{bmatrix}
\Psi^N(t)\\
\Phi^N(t)
\end{bmatrix}
=exp\biggl(
\begin{bmatrix}
\mathcal{A}^N & \partial\mathcal{A}^N/\partial\rho\\
0 & \mathcal{A}^N
\end{bmatrix}
\tau\biggr)
\begin{bmatrix}
0\\
I
\end{bmatrix}\label{eq4.18}
\end{align}
It follows from \eqref{eq4.18} that $\frac{\partial\hat{\mathcal{A}}^N}{\partial\rho}=\Psi^N(\tau)$. The details can be found in \cite{AUTO,JMAA}.

\section{Input Estimation or Deconvolution}\label{decon}

\indent

We formulate the problem of estimating or deconvolving the input to the population model, \eqref{eq3.8}-\eqref{eq3.10} (or in discrete time, \eqref{eq4.1}-\eqref{eq4.2}) as a constrained, regularized optimization problem that ultimately takes the form of a quadratic programming problem, or a linear quadratic tracking problem (see, for example, \cite{LJL}), albeit with a performance index whose evaluation requires the solution of an infinite dimensional state equation. In practice of course, as was the case in the distribution estimation problem discussed in the previous section, we solve a finite dimensional approximation. In this section, we prove the convergence of solutions of finite dimensional approximating problems to the solution of the original infinite dimensional estimation problem. In the next section we discuss how the approximating problems are regularized and how optimal values of the regularization parameters or weights can be determined.  In that section we also discuss how credible bands for the estimated or deconvolved input signal can be obtained based on the distribution of the random parameters in the population model.

We denote the parameters that determine the distribution of $\cq$ by $\rho$, where $\rho = (\vec{a},\vec{b},\vec{\theta})$, optimally fit parameters by $\rho^*$, where $\rho^* = (\vec{a}^*,\vec{b}^*,\vec{\theta}^*)$, and the support of the random parameters determined by the optimally fit parameters, $\rho^*$, by $Q^*$ (see Section 4 above).  We then consider the population model \eqref{eq4.1}-\eqref{eq4.2} in which the input $\{u_j\}$ is obtained by zero-order hold sampling a continuous time signal.  That is, we assume that the input to the population model is given by $\{\cu_j\}$ with $\cu_j = \cu(j\tau) \in L_{\pi(\rho^*)}^2(Q^*)$, where $\tau > 0$ is the length of the sampling interval and $\cu$ is a continuous time signal that is at least continuous (to allow for sampling) on an interval $[0,T]$. In our framework, we first use this model with the training data to estimate the distribution of the random parameters (i.e. to obtain $\rho^*$). Then, we seek an estimate for the input based on this model and the estimated distribution of the random parameters. We consider our input estimate, $\cu$, to be a random variable being a function of the random parameters, $\cq$, as well as a function of time. Therefore, let $\cu \in S(0,T)$, where $S(0,T)= H^1(0,T;L_{\pi(\rho^*)}^2(Q^*))$ and let $\cU$ be a compact subset of $S(0,T)$. 

The input estimation problem is then given by
\begin{align}
    \cu^*=\argmin_{\cU}J(\cu)=\argmin_{\cU}\sum_{k=1}^K\left|\cy_k(\cu)-\hat{y}_k\right|^2+||\cu||_{S(0,T)}^2,\label{eq5.1}
\end{align}
where the term $||\cu||_{S(0,T)}^2$ is related to the regularization with $||\cdot||_{S(0,T)}$ a norm on the space $S(0,T)$, the  details of which will be discussed later, and
\begin{align}
    \cy_k(\cu)=\sum_{j=0}^{k-1}\left<h_{k-j}(\rho^*),\cu_j\right>_{L_{\pi(\rho^*)}^2(Q^*)}, \ \ k=1,2,...,K,\label{eq5.2}
\end{align}
with $\cu_j=\cu(j\tau;q)$, $j=0,1,...,K$, $h_l(\rho^*)=\hat{\cC}(\rho^*)\hat{\cA}(\rho^*)^{l-1}\hat{\cB}(\rho^*)\in L_{\pi(\rho^*)}^2(Q^*)^*=L_{\pi(\rho^*)}^2(Q^*)$, $l=1,2,...,K$ where $\hat{\cC}(\rho^*)=\cC(\rho^*)\in\cL(\cV,\bR)$, $\hat{\cA}(\rho^*)=e^{\cA(\rho^*)\tau}\in\cL(\cV^*,\cV)$, and $\hat{\cB}(\rho^*)=\int_0^{\tau}e^{\cA(\rho^*)s}\cB(\rho^*)ds=\cA(\rho^*)^{-1}(\hat{\cA}(\rho^*)-I)\cB(\rho^*)\in\cL(L_{\pi(\rho^*)}^2(Q^*),\cV)$. We note that the coercivity assumption implies without loss of generality (via a standard change of variable), that we may assume that the operator $\cA(\rho^*)$ has a bounded inverse, $\cA(\rho^*)^{-1} \in\cL(\cV^*,\cV)$. It is also important to indicate here that, due to the assumption that $\cu$ is also a function of $\cq$, we consider the operator $\cB(\rho^*)$ now to be defined from $L_{\pi(\rho^*)}^2(Q^*)$ into $\cV^*$ as
\begin{align}
    \left<\cB(\rho^*)\cu,\psi\right>=\int_{Q^*}\left<B(q)\cu(q),\psi(q)\right>_{V^*,V}f(q;\rho^*)dq.\label{eq5.2a}
\end{align}

Solving the problem stated in \eqref{eq5.1} requires finite dimensional approximation. Let $N,M,L$ be multi-indices defined by $N=(n,m_1,m_2)$, $M=(m,m_1,m_2)$, and $L=(N,M)$. Note that whenever we use the notation $N$, $M$, or $L\ra\infty$, we mean all components of these multi-indices go to infinity. For each $M$, let $\cU^M$ be a closed subset of $\cU$ which is contained in a finite dimensional subspace of $S(0,T)$.  The sets  $\cU^M$ are, of course, therefore compact. 

We will require the following approximation assumption on the subsets $\cU^M$, or more typically, the finite dimensional spaces they are contained in, $\cS^M\subset S(0,T)$. \\

\noindent
\textbf{Assumption 5.1} For each $\cu\in\cU$, there exists a sequence $\{\cu^M\}$ with $\cu^M\in\cU^M$, such that
\begin{align}
    ||\cu^M-\cu||_{S(0,T)}\ra0, \ \ \text{as} \ \ M\ra\infty.\label{eq5.3} 
\end{align}
\\

\noindent
Let the Bochner spaces $\cV=L_{\pi(\rho^*)}^2(Q^*;V)$ and $\cH=L_{\pi(\rho^*)}^2(Q^*;H)$ which in the usual manner, form the Gelfand triple $\cV\hra\cH\hra\cV^*$, be as they were defined in Section 3. For each $N$ let $\cV^N$ denote a finite dimensional subspace of $\cV$ and let $\cP^N_\cH:\cH \ra \cV^N$ denote the orthogonal projection of $\cH$ onto $\cV^N$.  We require the assumption that the subspaces $\cV^N$ have the following approximation property. \\

\noindent
\textbf{Assumption 5.2} For each $\mathcal{v}\in\cV$
\begin{align}
    ||\cP^N_\cH\mathcal{v}-\mathcal{v}||_{\mathcal{V}}\ra0, \ \ \text{as} \ \ N\ra\infty.\label{eq5.3a} 
\end{align}
\\

\noindent
We note that it can often be shown using the Schmidt inequality (see \cite{Schultz1973}) that spline-based subspaces $\cV^N$ typically satisfy \eqref{eq5.3a}.

Define the operators $\cA^N(\rho^*):\cV^N\ra\cV^N$ by
\begin{align}
    \begin{aligned}
        \left<\cA^N(\rho^*)\cv^N,\cw^N\right>&=-\ca(\cv^N,\cw^N)\\
        &=-\int_{Q^*} a(q;\cv^N(q),\cw^N(q))dF^N(q)\\
        &=-\int_{Q^*} a(q;\cv^N(q),\cw^N(q))f(q;\vec{a}^*,\vec{b}^*,\vec{\theta}^*)dq,\label{eq5.3b}
    \end{aligned}
\end{align}
where $\cv^N,\cw^N\in\cV^N$.  We note that as is the case with the operator $\cA(\rho^*)$, coercivity implies that without loss of generality, the operators $\cA^N(\rho^*)$ given in \eqref{eq5.3b} are invertible with inverses that are uniformly bounded in $N$.  Moreover, once again, as a result of coercivity, standard estimates can be used to show that 
\begin{align}
    \begin{aligned}
        ||\cA^N(\rho^*)^{-1}\cP^N\varphi-\cP^N\cA(\rho^*)^{-1}\varphi||_\cV \ra 0, \varphi \in \cH\label{eq5.3c}
    \end{aligned}
\end{align}
as $N \ra \infty$ (see, for example, \cite{BK, BI, JMAA}). It then follows from the Trotter-Kato semigroup approximation theorem \cite{BK, BI, pazy83} that   
\begin{align}
    \begin{aligned}
        ||\hat{\cA}^N(\rho^*)\cP^N\varphi-\cP^N\hat{\cA}(\rho^*)\varphi||_\cV \ra 0, \varphi \in \cH\label{eq5.3d}
    \end{aligned}
\end{align}
as $N \ra \infty$.

We now state the finite dimensional approximating deconvolution problems as follows:
\begin{align}
    \cu_L^*=\argmin_{\cU^M}J^L(\cu)=\argmin_{\cU^M}\sum_{k=1}^K\left|\cy_k^N(\cu)-\hat{y}_k\right|^2+||\cu||_{S(0,T)}^2,\label{eq5.4}
\end{align}
where
\begin{align}
    \cy_k^N(\cu)=\sum_{j=0}^{k-1}\left<h_{k-j}^N(\rho^*),\cu_j\right>_{L_{\pi(\rho^*)}^2(Q^*)}, \ \ k=1,2,...,K,\label{eq5.5}
\end{align}
with $\cu_j=\cu(j\tau) \in L_{\pi(\rho^*)}^2(Q^*)$, $j=0,1,...,K$, $h_l^N(\rho^*)=\hat{\cC}^N(\rho^*)\hat{\cA}^N(\rho^*)^{l-1}\hat{\cB}^N(\rho^*)\in L_{\pi(\rho^*)}^2(Q^*)^*=L_{\pi(\rho^*)}^2(Q^*)$, $l=1,2,...,K$, $\hat{\cC}^N(\rho^*)=\cC(\rho^*)$, $\hat{\cA}^N(\rho^*)=e^{\cA^N(\rho^*)\tau}\in\cL(\cV^N,\cV^N)$, $\hat{\cB}^N(\rho^*)=\int_0^{\tau}e^{\cA^N(\rho^*)s}\cB^N(\rho^*)=\cA^N(\rho^*)^{-1}(\hat{\cA}^N(\rho^*)-I)\cB^N(\rho^*)$ with $\cB^N(\rho^*)=\bar{\cP}_{\cH}^N\cB(\rho^*) \in \cL(L_{\pi(\rho^*)}^2(Q^*),\cV^N)$, the operator $\cB(\rho^*)$ given by \eqref{eq5.2a}, and $\bar{\cP}_{\cH}^N$ being the extension of the orthogonal projection of $\cH$ onto $\cV^N$ to $\cV^*$. Note that the existence of the extension of the orthogonal projection operator can be proved by using the fact that $\cH$ is a dense subset of $\cV^*$ and the Riesz Representation theorem.

Now, consider the following lemma which we will need to prove our convergence theorem.

\begin{lemma}
For $l=1,2,...,K$, let the bounded linear functionals $h_l(\rho^*)$ and $h_l^N(\rho^*)$ on $L_{\pi(\rho^*)}^2(Q^*)$ be given by  
$h_l(\rho^*)=\cC(\rho^*)\hat{\cA}(\rho^*)^{l-1}\cA(\rho^*)^{-1}(\hat{\cA}(\rho^*)-I)\cB(\rho^*)$ and\\ $h_l^N(\rho^*)=\cC(\rho^*)\hat{\cA}^N(\rho^*)^{l-1}\cA^N(\rho^*)^{-1}(\hat{\cA}^N(\rho^*)-I)\cB^N(\rho^*)= \cC(\rho^*)\hat{\cA}^N(\rho^*)^{l-1}\cA^N(\rho^*)^{-1}$ $(\hat{\cA}^N(\rho^*)-I)\bar{\cP}_{\cH}^N\cB(\rho^*)$, respectively. Then, under Assumption 5.2, we have that $h_l^N(\rho^*)$ converges weakly (i.e. pointwise) to $h_l(\rho^*)$ in $L_{\pi(\rho^*)}^2(Q^*)^*$ (and therefore, their respective Riesz representers in $L_{\pi(\rho^*)}^2(Q^*)$ converge weakly in $L_{\pi(\rho^*)}^2(Q^*)$ as well) as $N \ra \infty$, uniformly in $l$ on any finite set of indices.  That is
\begin{align}
    |\left<h_l^N(\rho^*),\cw\right>_{L_{\pi(\rho^*)}^2(Q^*)}-\left<h_l(\rho^*),\cw\right>_{L_{\pi(\rho^*)}^2(Q^*)}| \ra 0,\label{eq5.6}
\end{align}
as $N\ra\infty$, for all $\cw \in L_{\pi(\rho^*)}^2(Q^*)$, $l=1,2,...,K$, uniformly in $l$ on any finite set of indices.  Moreover, if $\{\cu^M\}$ is a sequence in $L_{\pi(\rho^*)}^2(Q^*)$ with $\cu^M \ra \cu \in L_{\pi(\rho^*)}^2(Q^*)$, we have that $\left<h_l^N(\rho^*),\cu^M\right> \ra \left<h_l(\rho^*),\cu\right>$ as $L \ra \infty$, uniformly in $l$ on any finite set of indices.
\end{lemma}

\begin{proof}
The last claim in the statement of the Lemma of course follows immediately from the weak convergence in \eqref{eq5.6}.  To establish \eqref{eq5.6}, we note that in light of the approximation assumption, Assumption 5.2, on the finite dimensional subspaces $\cV^N$, \eqref{eq5.3a}, and the strong convergence in  \eqref{eq5.3c} and \eqref{eq5.3d}, the result will follow if we can show that the operators $\bar{\cP}_{\cH}^N$ converge strongly to the identity on the space $\cV^*$. But this follows easily from the approximation condition \eqref{eq5.3a} and standard density arguments.
\end{proof}

Now, we are ready to state and prove the the theorem regarding convergence of the solutions of approximating problems to the solution of the main problem.

\begin{theorem}
 For each $L$, the finite dimensional approximating optimization problem given in \eqref{eq5.4} admits a solution denoted by $\cu_L^*$. Moreover, under Assumptions 5.1 and 5.2 there exists a subsequence of $\{\cu_L^*\}$, $\{\cu_{L_k}^*\}\subset\{\cu_L^*\}$, with $\cu_{L_k}^*\ra\cu^*$ as $k\ra\infty$, with $\cu^*$ a solution to the infinite dimensional estimation problem given in \eqref{eq5.1}.
\end{theorem}

\begin{proof}
The existence of a solution to each of the finite dimensional approximating optimization problems given in \eqref{eq5.4} follows from \eqref{eq5.5} and therefore the continuity of $J^L$, and the compactness of $\cU^M$.  Now let $\{\cu_M\}$ be a convergent sequence in $\cU \subset S(0,T)$ with $\cu_M \in \cU^M$ and $||\cu_M-\cu||_{S(0,T)} \ra 0$ as $M \ra \infty$, $\cu \in \cU$. Then if $\cu_{M,j} = \cu_M(j\tau) \in L_{\pi(\rho^*)}^2(Q^*)$ and $\cu_j = \cu(j\tau) \in L_{\pi(\rho^*)}^2(Q^*)$, $j = 1,2, ..., K$, it follows that $||\cu_{M,j}-\cu_j||_{L_{\pi(\rho^*)}^2(Q^*))}\ra 0$ as $M \ra \infty$, $j = 1,2, ..., K$. It follows from the lemma that 
\begin{align}
    |\left<h_l^N(\rho^*),\cu_{M,j}\right>-\left<h_l(\rho^*),\cu_{j}\right>|\ra0, \ \ \text{as} \ \ L = (N,M) \ra\infty, \ \ j,l=1,2,...,K.\label{eq5.6a}
\end{align}
It follows from \eqref{eq5.6a} that $J^L(\cu_M) \ra J(\cu)$ as $L \ra \infty$.

Let $\{\cu_L^*\} \subset \cU_M \subset \cU$ be a sequence of solutions to the finite dimensional approximating optimization problems given in \eqref{eq5.4}.  The compactness of $\cU$ implies that there exists a subsequence of $\{\cu_L^*\}$, $\{\cu_{L_k}^*\}\subset\{\cu_L^*\}$, with $\cu_{L_k}^*\ra \cu^* \in \cU$ as $k \ra \infty$.  Then for any $\cu \in \cU$ it follows that
\begin{align}
    \begin{aligned}
        J(\cu^*)=J(\lim_{k \ra \infty}\cu_{L_k}^*)=\lim_{k \ra \infty} J^{L_k}(\cu_{L_k}^*)\leq \lim_{k \ra \infty} J^{L_k}(\cu^{M_k})=J(\lim_{k \ra \infty}\cu^{M_k})=J(\cu), \label{eq5.6b}
    \end{aligned}
\end{align}
where $L_k=(N_k,M_k)$, and the sequence $\{\cu^{M_k}\} \subset \cU$ with $\cu^{M_k} \in \cU^{M_k}$ are the approximations to $\cu$ guaranteed to exist by the approximation assumption on the subsets $\cU^M$ given in Assumption 5.1, equation \eqref{eq5.3}, and thus, $\cu^*$ is a solution of the estimation problem stated in \eqref{eq5.1}.
\end{proof}

We note first that the above theorem continues to hold if $S(0,T)$ is chosen as $S(0,T)= L_2(0,T;L_{\pi(\rho^*)}^2(Q^*))$.  In addition, it is not difficult to see that the way we have formulated the deconvolution problem and its finite dimensional approximations given in \eqref{eq5.1} and \eqref{eq5.4}, respectively, the resulting optimization problems take the form of a linear quadratic tracking problem where the systems to be controlled are, respectively, the population model, \eqref{eq4.1},\eqref{eq4.2} and its finite dimensional approximation, \eqref{eq4.9}, \eqref{eq4.10}.  Consequently, it becomes possible to adapt some ideas from \cite{BB} to obtain a somewhat different and in some ways stronger convergence result than that given in Theorem 5.1 above. 

Now let $\cU$ be a closed and convex subset of $S(0,T)$ and for each $M$, let $\cU^M$ be a closed convex subset of $\cU$ which is contained in a finite dimensional subspace of $S(0,T)$. Note that the maps $\cu \mapsto J(\cu)$ and $\cu \mapsto J^L(\cu)$ from $\cU$ into $\mathbb{R}$ and from $\cU^M$ into $\mathbb{R}$, respectively, are strictly convex.   Consequently solutions $\cu^*$ and $\cu_L^*$ to the optimization problems \eqref{eq5.1} and \eqref{eq5.4}, respectively, exist and are unique.  Moreover, the sequence $\{||\cu_L^*||_{S(0,T)}\}$ is bounded. If not, there would exist a subsequence, $\{\cu_{L_k}^*\}$ with $||\cu_{L_k}^*||_{S(0,T)} \rightarrow \infty$, and therefore that $J^{L_k}(\cu_{L_k}^*) \rightarrow \infty$.  But then this would contradict the fact that for any $u \in \cU$, we have that

\begin{align}
    \begin{aligned}
        J^{L_k}(\cu_{L_k}^*)\leq J^{L_k}(\cu^{M_k}) \rightarrow J(\cu) < \infty, \label{eq5.6c}
    \end{aligned}
\end{align}
where $L_k=(N_k,M_k)$, and the sequence $\{\cu^{M_k}\} \subset \cU$ with $\cu^{M_k} \in \cU^{M_k}$ are the approximations to $\cu$ guaranteed to exist by the approximation assumption on the subsets $\cU^M$ given in Assumption 5.1, equation \eqref{eq5.3}.  Then since $\cU$ was assumed to be closed and convex, it is weakly closed.  Therefore, there exist a weakly convergent subsequence, $\{\cu_{L_k}^*\}$ of $\{\cu_L^*\}$ with $\cu_{L_k}^* \rightharpoonup \cu^*$, for some $\cu^* \in \cU$.  In addition, the convexity and (lower semi-) continuity of $J$ and $J^L$ yield that $J$ and $J^L$ are weakly sequentially lower semi-continuous.  Then, Assumptions 5.1 and 5.2, Lemma 5.1 together with the weak lower semicontinuity of $J$ and $J^L$ imply that
\begin{align}
    \begin{aligned}
        J(\cu^*)=J(w-\lim_{k \ra \infty}\cu_{L_k}^*) \leq \liminf_{k \ra \infty} J(\cu_{L_k}^*) = \liminf_{k \ra \infty} J^{L_k}(\cu_{L_k}^*) \\ \leq \limsup_{k \ra \infty} J^{L_k}(\cu_{L_k}^*) \leq \limsup_{k \ra \infty} J^{L_k}(\cu^{M_k}) = J(\lim_{k \ra \infty}\cu^{M_k})=J(\cu), \label{eq5.6d}
    \end{aligned}
\end{align}
where again $L_k=(N_k,M_k)$, and the sequence $\{\cu^{M_k}\} \subset \cU$ with $\cu^{M_k} \in \cU^{M_k}$ are the approximations to $\cu$ guaranteed to exist by the approximation assumption on the subsets $\cU^M$ given in \eqref{eq5.3} (Assumption 5.1), and hence that, $\cu^*$ is the solution of the estimation problem stated in \eqref{eq5.1}.  Finally we note that the strict convexity of  $J(\cu)$ and $J^L(\cu)$ imply that the sequence itself, $\{\cu_L^*\}$ must in fact converge weakly to the unique solution, $\cu^*$, of the estimation problem stated in \eqref{eq5.1}, and because the norm in $S(0,T)$ is bounded above by $J$, that the sequence $\{\cu_L^*\}$ itself must in fact converge strongly, or in the norm in $S(0,T)$, to the unique solution, $\cu^*$, of the estimation problem stated in \eqref{eq5.1}.

\section{Matrix Representations of the Operators and Regularization}\label{Reg}

\indent

In this section we describe how the matrix representations for the finite dimensional operators that appear in the approximating population model, \eqref{eq4.9}, \eqref{eq4.10}, are computed. We also discuss some issues related to the regularization terms that appear in the optimization problems \eqref{eq5.1} and \eqref{eq5.4}. The state, $\cx_{i,j}$, $j=0,1,...,\mu_i$, $i=1,2,...,\nu$, in the population model described by system \eqref{eq4.1}-\eqref{eq4.2}, is a function of the distributed variables, $\eta$, and the random parameters $q_1$ and $q_2$. It is these dependencies that are discretized in our framework. Recall that the state space for the approximating systems \eqref{eq4.9}-\eqref{eq4.10} is the finite dimensional subspace $\cV^N$ of $\cV$. Recall also that these spaces must satisfy the approximation property given in Assumption 5.2, equation \eqref{eq5.3a}. We construct the spaces $\cV^N$ using linear B-splines $\{\varphi_j^n\}_{j=0}^n$ defined with respect to the uniform mesh $\{j/n\}_{j=0}^n$ on the interval $[0,1]$ and use the standard $0^{th}$ order (i.e. piecewise constants) B-splines  $\{\chi_{1,j}^{m_1}\}_{j=1}^{m_1}$ and $\{\chi_{2,j}^{m_2}\}_{j=1}^{m_2}$ defined with respect to the uniform mesh $\{a_i^*+(b_i^*-a_i^*)j/m_i\}_{j=0}^{m_i}$ on the intervals $[a_i^*,b_i^*]$ for $i=1,2$, respectively. Then, letting $N=(n,m_1,m_2)$ as before, and letting $P$ denote the multi-index $P=(j,j_1,j_2)$, we use tensor products to define $\{\psi_P^N\}_P^N$ as $\psi_P^N=\varphi_j^n\chi_{1,j_1}^{m_1}\chi_{2,j_2}^{m_2}$ and set $\cV^N=span\{\psi_P^N\}_P^N$.  Results from \cite{Schultz1973} can be used to establish equation \eqref{eq5.3a} of Assumption 5.2. 

The input signal, $\cu$, is a function of time, $t$, and also the random parameters $q_1$ and $q_2$ and is an element of the space $S(0,T)$ defined in Section 5.  To construct finite dimensional subspaces, $\cS^M$, of $S(0,T)$ that have the approximation property given in equation \eqref{eq5.3} of Assumption 5.1 (see \cite{Schultz1973}), we again use standard linear B-spline polynomials $\{\phi_i^m\}_{i=1}^m$ defined with respect to the usual uniform mesh, $\{iT/m\}_{i=0}^m$, on the interval $[0,T]$ to discretize the time dependency. We use the same $0^{th}$ order B-splines $\{\chi_{1,j}^{m_1}\}_{j=1}^{m_1}$, $\{\chi_{2,j}^{m_2}\}_{j=1}^{m_2}$ for $q_1$ and $q_2$ with respect to corresponding uniform mesh on the intervals $[a_1^*,b_1^*]$ and $[a_2^*,b_2^*]$, respectively, to discretize the dependence on the random parameters. Letting $M=(m,m_1.m_2)$ and $R$ be the multi-index $R=(i,j_1,j_2)$, we again use tensor products to serve as the basis for the approximating subspaces.  We define $\{\xi_R^M\}_R^M$ as $\xi_R^M=\phi_i^m\chi_{1,j_1}^{m_1}\chi_{2,j_2}^{m_2}$ and set $\cS^M=span\{\xi_R^M\}_R^M$. In this way, any $\cu \in\cS^M$ can be written as
\begin{align}
    \cu(t;q)=\sum_{i=1}^m\sum_{j_1=1}^{m_1}\sum_{j_2=1}^{m_2}\cu^M_{i,j_1,j_2}\phi_i^m(t)\chi_{1,j_1}^{m_1}(q_1)\chi_{2,j_2}^{m_2}(q_2).\label{eq6.1}
\end{align}
Now that bases for the approximating spaces have been chosen, the definitions of the finite dimensional operators $\cA^N(\rho^*)$, $\cB^N(\rho^*)$, $\cC^N(\rho^*)$, $\hat{\cA}^N(\rho^*)$, $\hat{\cB}^N(\rho^*)$, and $\hat{\cC}^N(\rho^*)$ given in Section 5 are sufficient to compute corresponding matrix representations. Note that since the finite dimensional input operators act on the input signal $\cu$ which depends on $t$, and have range in $\cV^N$, their matrix representations will depend on the dimension of both $\cV^N$ and $\cS^M$.  Consequently their matrix representations will depend on the multi-index $L = (N,M)$. In this way, the matrix representation of the system \eqref{eq4.9}-\eqref{eq4.10} is of the form 
\begin{align}
    \begin{aligned}
        \mathbb{M}^N\mathbb{X}_{k+1}^N&=\mathbb{K}^N\mathbb{X}_k^N + \mathbb{B}^L\mathbb{U}_k^M,\ k=0,1,2,...,K, \\
        \cy_k^L&=\mathbb{C}^N\mathbb{X}_k^N,\ k=0,1,2,...,K,\label{eq6.2}
    \end{aligned}    
\end{align}
where the vectors $\mathbb{X}_k^N$ are coefficients of the basis elements $\{\psi_P^N\}_P^N$ and the vectors $\mathbb{U}_k^M$ are formed from the coefficients $\cu^M_{i,j_1,j_2}$ in the expansion given in \eqref{eq6.1}.  In the system \eqref{eq6.2} the matrices $\mathbb{M}^N$ and $\mathbb{K}^N$ are $(n+1)m_1m_2\times(n+1)m_1m_2$, $\mathbb{B}^L$ is an $(n+1)m_1m_1\times mm_1m_2$ matrix, and $\mathbb{C}^N$ is an $(n+1)m_1m_1$ dimensional row vector. The state variables $\mathbb{X}_k^N$ are $(n+1)m_1m_2\times1$ vectors and the inputs $\mathbb{U}_k^M$ are $mm_1m_2\times1$ vectors. The matrix representation of the operators $\cA^N$, $\hat{\cA}^N$, $\hat{\cB}^N$, and $\hat{\cC}^N$ in terms of the matrices $\mathbb{M}^N$, $\mathbb{K}^N$, $\mathbb{B}^L$, and $\mathbb{C}^N$ appearing in \eqref{eq6.2} can be used to obtain the matrix representations for the approximating convolution kernels $h_l^N(\rho^*)$, $l=1,2,...,K$, which because of their dependence on $\mathbb{B}^L$, we shall now refer to as $h_l^L(\rho^*)$. In the numerical studies presented in the next section the sampling interval was taken to be $\tau=1$ min.  It follows that $T/\tau=K$ in \eqref{eq5.1} and \eqref{eq5.4}. In the case that the sensor data, $\{\hat{y}_k\}$, has not been collected at 1-min time intervals, we re-sample by interpolating the collected data with a cubic spline.

The cost functionals defined in \eqref{eq5.1} and \eqref{eq5.4} include regularization terms in the form of the square of the norm on $S(0,T)$, $||\cdot||_{S(0,T)}$, which we take to be given by
\begin{align}
    \cR(r_1,r_2)=||u||_{S(0,T)}=\left(r_1\int_0^T||\cu(t)||_{L_{\pi(\rho^*)}^2(Q^*)}^2dt + r_2\int_0^T||\dot{\cu}(t)||_{L_{\pi(\rho^*)}^2(Q^*)}^2dt\right)^{1/2},\label{eq6.3}
\end{align}
where $r_1,r_2>0$ are regularization parameters in the form of nonnegative weights. While choosing regularization parameters can entail a mix of art and science, we chose $r_1$ and $r_2$ as the solution of the following optimization problem based on the original training data used to estimate the distribution of the random parameters in the model.  

Recall the training data from Section 4 consisting of $\nu$ data sets
\begin{align}
(\tilde{u}_i,\tilde{y}_i)_{i=1}^{\nu}=\left(\{\tilde{u}_{i,j}\}_{j=0}^{\mu_i-1},\{\tilde{y}_{i,j}\}_{j=1}^{\mu_i}\right)_{i=1}^{\nu}.
\label{eq6.4}
\end{align}
After using the training data to estimate the distribution of the random parameters, we use the approach discussed in Section 5 together with the training TAC observations, $\{\tilde{y}_{i,j}\}_{j=1}^{\mu_i}$, given in \eqref{eq6.4} to estimate the BrAC inputs $\{\tilde{u}_{i,j}\}_{j=0}^{\mu_i-1}$ also given in \eqref{eq6.4}. Denote these estimates by $\{\tilde{\cu}_{L;i,j}^{*}\}_{j=0}^{\mu_i-1}$.  We then use these estimates of the input together with the approximating population model, \eqref{eq5.2}, to obtain estimates for the corresponding TAC, $\{\tilde{y}_{L;i,j}^{*}\}_{j=1}^{\mu_i}$. Note that both the $\{\tilde{\cu}_{L;i,j}^{*}\}_{j=0}^{\mu_i-1}$ and the $\{\tilde{y}_{L;i,j}^{*}\}_{j=1}^{\mu_i}$ are functions of the weights, $r_1$ and $r_2$ appearing in the cost functional for the input estimation problem.  We then define the performance index   
\begin{align}
    J(r_1,r_2)=\sum_{i=1}^{\nu}\sum_{j=1}^{\mu_i}\left(|\bar{u}_{L;i,j-1}^{*}-\tilde{u}_{i,j-1}|^2+|\tilde{y}_{L;i,j}^{*}-\tilde{y}_{i,j}|^2\right),\label{eq6.5}
\end{align}
where $\bar{u}_{L;i,j}^{*}=\mathbb{E}[\tilde{\cu}_{L;i,j}^{*}|\pi(\rho^*)]$, and choose the regularization parameters, $r_1^*$ and $r_2^*$, to use in our scheme as
\begin{align}
    (r_1^*,r_2^*)=\argmin_{{\mathbb{R}^{+}}^2}J(r_1,r_2),\label{eq6.6}
\end{align}
where $J(r_1,r_2)$ is given by \eqref{eq6.5}.

We note that since the approximating optimization problems given in \eqref{eq5.4} are in fact constrained (i.e. we want all the components of $\mathbb{U}_k^M$ to be nonnegative) linear problems, in practice, we solve them as quadratic programming problems.  Moreover, in finite dimensions, the two integrals in the expression for the regularization term, $\cR(r_1,r_2)$, given in \eqref{eq6.3}, become quadratic forms in the vectors $\mathbb{U}_k^M$ with the two positive definite symmetric matrices, denoted by $\mathbb{Q}_1^M$ and $\mathbb{Q}_2^M$, having entries that are the easily computed $L_{\pi(\rho^*)}^2(Q^*)$ inner products of the basis elements for the $q$ dependencies of the elements in $\cS^M$.  By appropriately placing the matrix representations for the approximating filters $\{h_l^L(\rho^*)\}$ in the block matrix $\mathbb{H}^L$, the approximating optimization problem given in \eqref{eq5.4} now takes the equivalent form
\begin{align}
    \cu_L^*=\argmin_{\cU^M}J^L(\cu)=\argmin_{\cU^M}\left|\left|\begin{bmatrix}\mathbb{H}^L\\ (r_1^*\mathbb{Q}_1^M + r_2^*\mathbb{Q}_2^M)^{\frac{1}{2}}\end{bmatrix}\mathbb{U}^M - \mathbb{Y}^M\right|\right|_{\mathbb{R}^{K+Kmm_1m_2}}^2,\label{eq6.7}
\end{align}
where $(r_1^*,r_2^*)$ come from \eqref{eq6.6}, $\mathbb{U}^M$ is the $Kmm_1m_2$ dimensional column vector of the coefficients of the $\cu \in \cU^M$, and $\mathbb{Y}^M$ is the $K+Kmm_1m_2$ dimensional column vector consisting of the $K$ TAC data points $\{\hat{y}_k\}$ followed by $Kmm_1m_2$ zeros. The optimization problem given in \eqref{eq6.7} with the constraint $\mathbb{U}^M \geq 0$ is readily solved using the MATLAB routine LSQNONNEG.

\section{Numerical Results}\label{num.res}

\indent

In this section, we present our numerical results for the estimation of BrAC using our mathematical framework with actual alcohol biosensor data.  Specifically, we consider two examples. In the first example, the data was collected from a single individual over multiple drinking episodes while in the second, we use data collected from multiple subjects each providing a single drinking episode. Both datasets contain both TAC and BrAC measured simultaneously (albeit at different sampling rates) by alcohol biosensor devices and breath analyzers, respectively, during drinking episodes conducted both in the laboratory and out in the field (in Example 7.2 below, only the data collected in the laboratory were used becasue these were the only ones for which we were provided both TAC and BrAC). This way the data can be used to both fit a population model and, using cross-validation techniques, test our scheme's ability to deconvolve or recover BrAC in the form of eBrAC from the TAC signal.  Note here that a $\text{WrisTAS}^{\text{TM}}7$ alcohol biosensor designed and manufactured by Giner, Inc. of Waltham, MA was used for TAC measurements in the first dataset, while the device used in the second dataset was a Secure Continuous Alcohol Monitoring System (SCRAM) device manufactured by Alcohol Monitoring Systems in Littleton, Colorado (see Figure \ref{fig1}).
\begin{figure}[ht]
\centering
\includegraphics[scale=0.6]{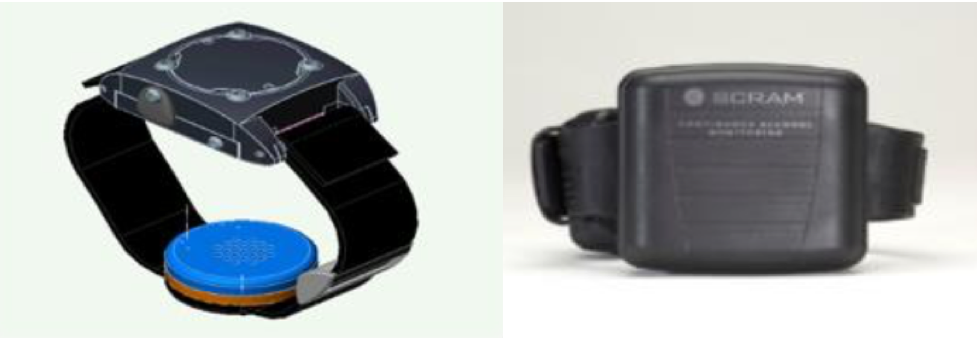}
\caption{Alcohol Biosensor Devices: The WrisTAS (left) and the SCRAM (right).}
\label{fig1}
\end{figure}

Before discussing our estimation results, we describe how we obtained the population model (these results are described in detail in \cite{AUTO}).  We assumed that the joint pdf of the two (now assumed to be random) parameters that appear in the model \eqref{eq2.6}-\eqref{eq2.10} have compact support and takes the form of a truncated exponential family of distributions (see \cite{JMAA}). Recalling that ${q}=(q_1,q_2)$ (or, more precisely, ${\cq}=(\cq_1,\cq_2)$) represents the random parameters, we let their supports be determined by the four parameters in the two vectors, $\vec{a}=(a_1,a_2)$, $\vec{b}=(b_1,b_2)$, as $[a_1,b_1]\times[a_2,b_2]$. In both examples we assumed that the distribution of the parameters was a truncated bivariate normal with mean $\mu$ and covariance matrix $\Sigma$ and joint pdf denoted by $\phi(\cdot;\vec{\mu},\Sigma)$. Then, we set
\begin{align}
    f(q;\vec{a},\vec{b},\mu,\Sigma)=\frac{\phi(q;\mu,\Sigma)\chi_{[a_1,b_1]\times[a_2,b_2]}(q)}{\int_{a_2}^{b_2}\int_{a_1}^{b_1}\phi(q;\mu,\Sigma)dq_1dq_2}.\label{eq7.1}
\end{align}

In each example, a portion of the data (henceforth referred to as the population or training data) was used to fit the population model and a portion was held back for cross validation. The training data was first used to estimate $\vec{a}$, $\vec{b}$, $\mu$, and $\Sigma$ (see \cite{AUTO} for details). Then, using the population model constructed using this fit distribution, we applied our scheme to both the population and cross validation data to estimate BrAC from the measured TAC. Furthermore, in the studies described below, we found it useful to stratify the data before attempting to fit the distribution of the random parameters.  Following training the population model, any new subjects would be placed in one of the strata and then the appropriate fit distribution for that strata would be used when deconvolving the BrAC from the TAC measurements. In the two examples we present here, since our goal was simply to demonstrate proof-of-concept and the basic efficacy of our approach, we simply stratified the population data and the cross validation data based on prior knowledge of both the BrAC and the TAC.  In practice, of course, this is unrealistic in that once the training of the model on the population has been completed using both BrAC and TAC, for any new subjects, only the TAC would be known and available.  We are currently looking at the effect on our approach of stratifying based on readily observable, identifiable, and determinable characteristics (covariates) of a subject such has height, weight, body mass index (BMI), sex, drinking behavior (heavy, teetotaler, etc.), genetic profile, ethnicity, and so on. These studies are currently well underway and will be reported on elsewhere in the near future. 

In addition, we also obtained 75$\%$ credible bands for our BrAC estimates. To do this, we generated 1,000 random samples from the population distribution (i.e. the fit truncated bivariate normal), and selected the ones lying in the circular region $R\subset[a_1,b_1]\times[a_2,b_2]$ centered at the mean where $R$ was chosen so that $\iint\limits_{R}f(q;\vec{a},\vec{b},\mu,\Sigma)dq=0.75$. Since our scheme provides an estimate of BrAC as a function of $q=(q_1,q_2)$ (and of course of time, as well), we could obtain the corresponding samples of the BrAC by simply evaluating this function of $\cq=(\cq_1,\cq_2)$ at the samples of the distribution.  To actually obtain the credible bands, at each time $t_j$, we identified the maximum and minimum value of the BrAC among all the samples thus providing 75$\%$ likely upper and lower bounds, respectively, for the BrAC estimates.  
 
In addition, for comparison, we also examined the case where the estimated input is considered to be a function of time only (and therefore not of $q=(q_1,q_2)$ or $\cq=(\cq_1,\cq_2)$). In this case, we used the $\cB$ operator as it was defined in \eqref{eq3.13} and followed the same steps with regard to the data as in the case $\cu$ depending on both $t$ and $q$. However, in this case, the credible bands had to be determined via simulation.  That is, for each sample from the distribution substituted into the model equations, \eqref{eq2.6}-\eqref{eq2.10}, the optimal estimate for the BrAC was found deterministically as in \cite{Rosen2014}. Then the credible bands at each time $t_j$ were found as they were in the previous case by identifying the maximum and minimum value of the BrAC among all the samples to obtain 75$\%$ likely upper and lower bounds for the BrAC estimates.  
 
All computations were carried out in MATLAB on either MAC or PC laptops or desktops. To solve optimization problems \eqref{eq4.8} that result in the estimate for the distribution of the parameters, we used the MATLAB Optimization Toolbox constrained optimization routine FMINCON which requires the computation of the gradient of the cost functional with respect to the parameters.  We used both the adjoint method together with the sensitivity equations \eqref{eq4.14}-\eqref{eq4.18},  and finite differencing, independently, and both methods yielded the same result. Since the only constriant on the BrAC signal is that it be nonnegative (we do not actually enforce the constraint that the admissible input set be compact), we used the MATLAB routine LSQNONNEG to solve the non-negatively constrained linear system problems \eqref{eq5.4} or \eqref{eq6.7} associated with the deconvolution of the BrAC estimates. After estimating the distribution of the random parameters, we found the optimal regularization parameters $r_1^*$ and $r_2^*$ by solving the optimization problem given in \eqref{eq6.5} or \eqref{eq6.6} using the MATLAB routine FMINSEARCH.  In both of the examples that follow we took $n=m_1=m_2=4$ and $m = 6T_h$, where $T_h$ is the number of hours of TAC data available for the drinking episode(s) being analyzed.
 
In estimating the parameters for the distributions, each BrAC/TAC dataset in population or training data was used to find estimates for the parameters, $q=(q_1,q_2)$ via deterministic nonlinear least squares.  Sample means and covariances of these estimates were then computed and used as initial guesses or estimates when the optimization problem given in \eqref{eq4.8} was solved.

\subsection{Example: Single Subject, Multiple Drinking Episodes; Data Collected with the \texorpdfstring{$WrisTAS^{TM}7$}{} Alcohol Biosensor}\label{ex1}
 
\indent
 
One of the co-authors of this paper (S. E. L.) wore a $\text{WrisTAS}^{\text{TM}}7$ alcohol biosensor device for 18 days. During each drinking episode, she collected BrAC data (i.e. blew into a breath analyzer) approximately every 30 minutes. The first drinking episode was conducted in the laboratory and BrAC was measured every 15 minutes until it returned to 0.000. Then she wore the biosensor device for the following 17 days and consumed alcohol ad libitum. During those days, BrAC was measured every 30 minutes starting from the beginning of the drinking session until its value returned to 0.000. The $\text{WrisTAS}^{\text{TM}}7$ measured and recorded ethanol level at the skin surface every 5-minutes.  It is important to note that during those 17 days, the data was collected in a naturalistic setting.
\begin{figure}[ht]
\centering
\includegraphics[scale=0.3]{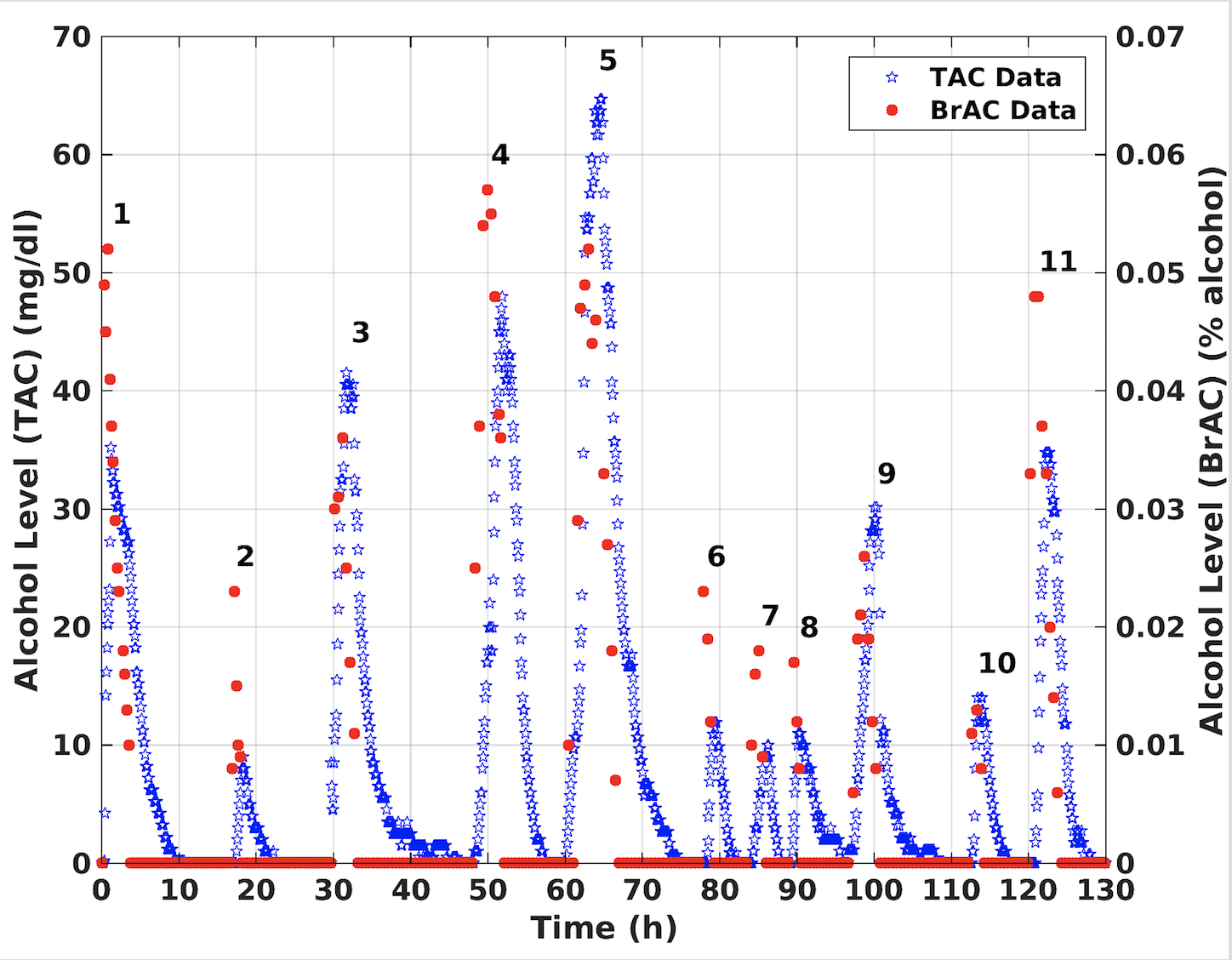}
\caption{BrAC and TAC measurements for Example 7.1.}
\label{fig2}
\end{figure}

The plot in Figure \ref{fig2} shows the measured BrAC and TAC over 11 drinking episodes. Using this plot, we visually stratified the dataset into two groups. The first group contains the data belonging to drinking episodes 1, 2, 4, 6, 7, 8, 11. In each of these drinking episodes, the peak BrAC value was higher than the bench calibrated peak TAC value. The second group contains the remaining drinking episodes, 3, 5, 9, and 10 for which the peak TAC value was higher than the peak BrAC value. In the example we present here, we considered only the first group.  Our results for the second group and all 11 drinking episodes taken together all at once were similar.  In the first group, we randomly chose the drinking episodes 4 and 8 for cross validation and used the remaining five, epsiodes 1, 2, 6, 7, and 11, to train the population model (i.e. these five episodes served as the training or population data).  We obtained the optimal parameters for the truncated bivariate normal distribution as follows: the truncated support was determined to be $a_1^*=0$, $b_1^*=1.4942$ and $a_2^*=0$, $b_2^*=2.0409$, and the optimal values for the mean and covariance matrix were found to be
\begin{align*}
    \mu^*=\begin{bmatrix}
    0.6245\\
    1.0274
    \end{bmatrix}
    \ \ \ \ \text{and} \ \ \ \
    \Sigma^*=\begin{bmatrix}
    0.0259 & 0.0067\\
    0.0067 & 0.1227
    \end{bmatrix}.
\end{align*}
The plot of the optimal joint density function corresponding to these optimal parameters is shown in Figure \ref{fig3}. In addition, the optimal regularization parameters were found to have the values $r_1^*=0.1591$ and $r_2^*=0.6516$ for the regularization term in \eqref{eq5.1} as defined in \eqref{eq6.3} in the case the input depends on both $t$ and $q$. Those values became $r_1^*=0.0000$ and $r_2^*=0.9653$ in the case $u$ depends only on $t$.
\begin{figure}[ht]
\centering
\includegraphics[scale=0.2]{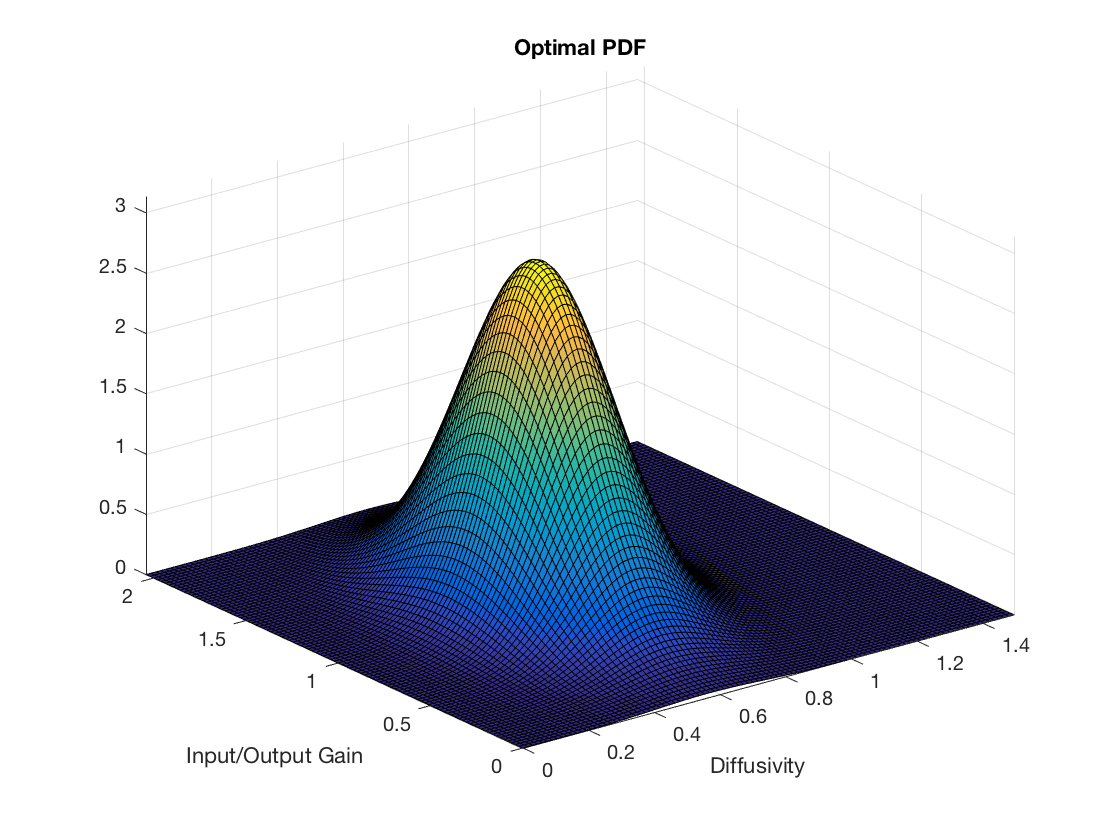}
\caption{Results for Example 7.1: Optimal pdf obtained using the data from drinking episodes 1, 2, 6, 7, and 11 as the population or training dataset.}
\label{fig3}
\end{figure}

As described previously in this section, we took two different approaches in estimating the input: (1) we sought $\cu$ is a function of both $t$ and $\cq$ and (2) we assumed that $u$ was a function of $t$ only. The estimation results for the training drinking episodes 1, 2, 6, 7, 11 that are obtained by each of the methods separately can be seen in Figure \ref{fig4}. The cross validation results for the drinking episodes 4 and 8 are shown in Figure \ref{fig5}.
\begin{figure}[ht]
\centering
\includegraphics[scale=0.32]{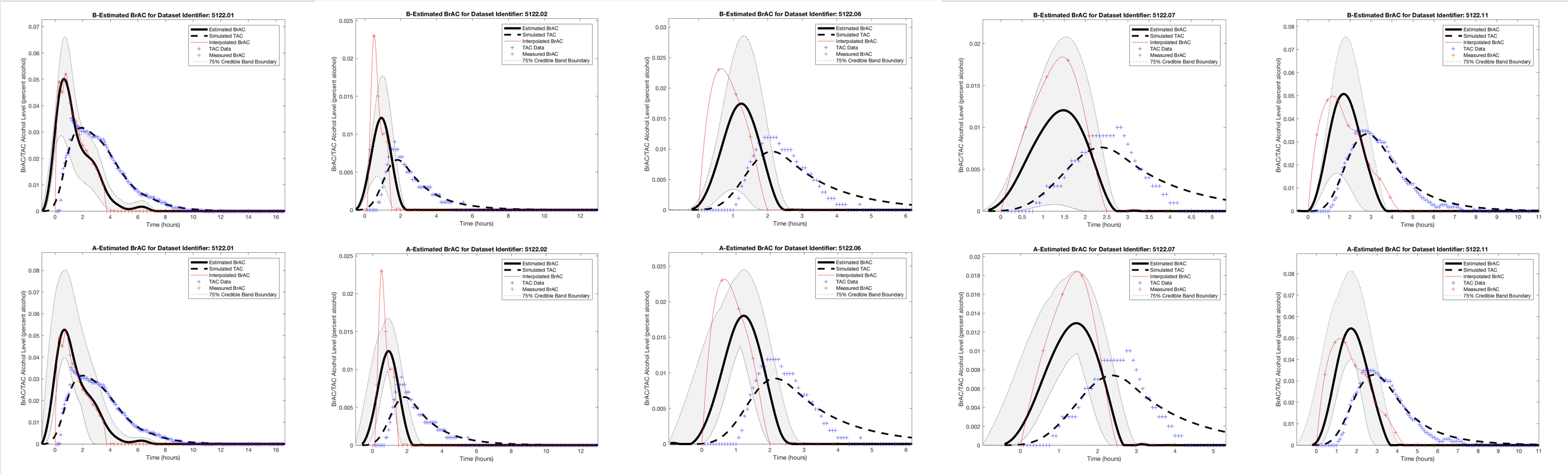}
\caption{Results for Example 7.1: Top row: BrAC estimates for the training datasets assuming the input, $\cu$, depends on both time, $t$, and the random parameters, $\cq$. Bottom row: BrAC estimates for the training datasets assuming the input, $u$, depends only on time, $t$.}
\label{fig4}
\end{figure}
\begin{figure}[b]
\centering
\includegraphics[scale=0.3]{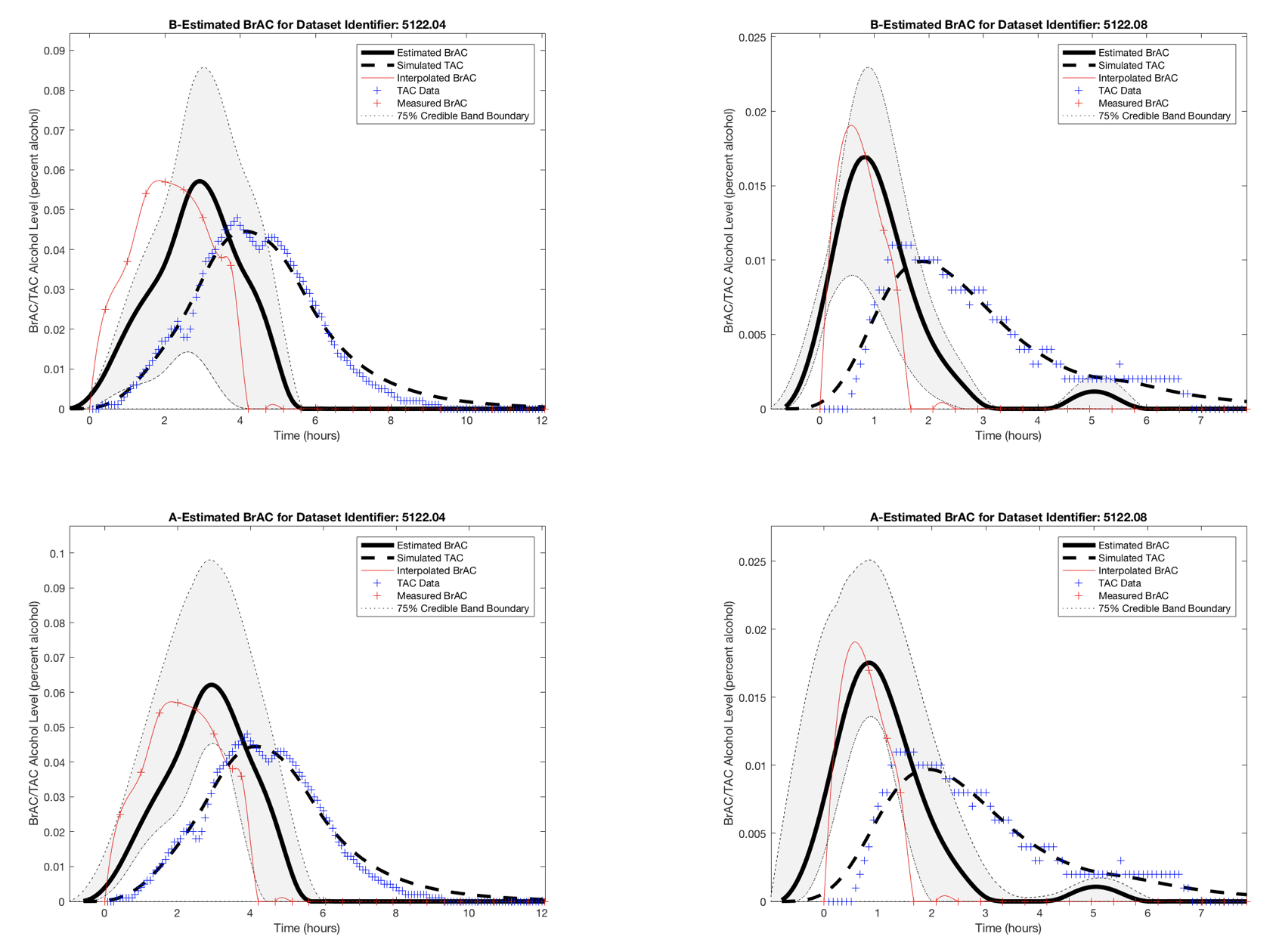}
\caption{Results for Example 7.1: Top row: BrAC estimates for the cross validation datasets assuming the input, $\cu$, depends on both time, $t$, and the random parameters, $\cq$. Bottom row: BrAC estimates for the training datasets assuming the input, $u$, depends only on time, $t$.}
\label{fig5}
\end{figure}

In the upper left hand panel of Figure \ref{fig11} we have plotted the mean or expected value of the approximating impulse response function or convolution kernel for the population model, given by
\begin{align}
    \begin{aligned}
        h_l^N(\rho^*)&=\cC(\rho^*)\hat{\cA}^N(\rho^*)^{l-1}\cA^N(\rho^*)^{-1}(\hat{\cA}^N(\rho^*)-I)\cB^N(\rho^*) \\
        &=\cC(\rho^*)\hat{\cA}^N(\rho^*)^{l-1}\cA^N(\rho^*)^{-1}(\hat{\cA}^N(\rho^*)-I)\bar{\cP}_{\cH}^N\cB(\rho^*), \label{eq7.1a}
    \end{aligned}    
\end{align}
$l=1,2,...,K$, together with the 75\% credible band.  Recall that for each $l=1,2,...,K$, $h_l^N(\rho^*)$, or at least its representer, is an element $L_{\pi(\rho^*)}^2(Q^*)$ and thus is a function of $\cq$, and consequently, the mean and credible bands can be obtained by simply substituting in samples of $\cq$. Note that as a result of the fact that the bases for our approximating subspaces are tensor products, the impulse response function for the scheme in which we simply sought an estimate for BrAC that was a function of $t$ only (and not of $q_1$ and $q_2$), turn out to be the mean of $h_l^N(\rho^*)$, $l=1,2,...,K$, $E[h_l^N(\rho^*)|\pi(\rho^*)]$, and as such are plotted in this figure as well. In the upper right hand panel of Figure \ref{fig11}, the surface, $h_l^N(\rho^*)$, as a function of $q_1$ and $q_2$ at the time at which $E[h_l^N(\rho^*)|\pi(\rho^*)]$ is at its peak, is plotted. In the lower left hand panel of Figure \ref{fig11} we have plotted the estimated BrAC for drinking episode 1 as a function of $q_1$ and $q_2$ at the time at which its mean is at its peak, and finally in the lower right hand panel of Figure \ref{fig11} the probability density function for the estimated BrAC at the time it is at its peak is plotted. Once again we note that since the estimated BrAC at each time $t$ is an element of $L_{\pi(\rho^*)}^2(Q^*)$ and thus is a function of $\cq$, this pdf can be obtained by simply generating samples of $\cq$ from the bivariate normal distribution determined by the parameters $\rho*$ and then simply substituting them into the obtained expression for the estimated BrAC, or eBrAC, \eqref{eq6.1}.     
\begin{figure}[ht]
\centering
\includegraphics[scale=0.5]{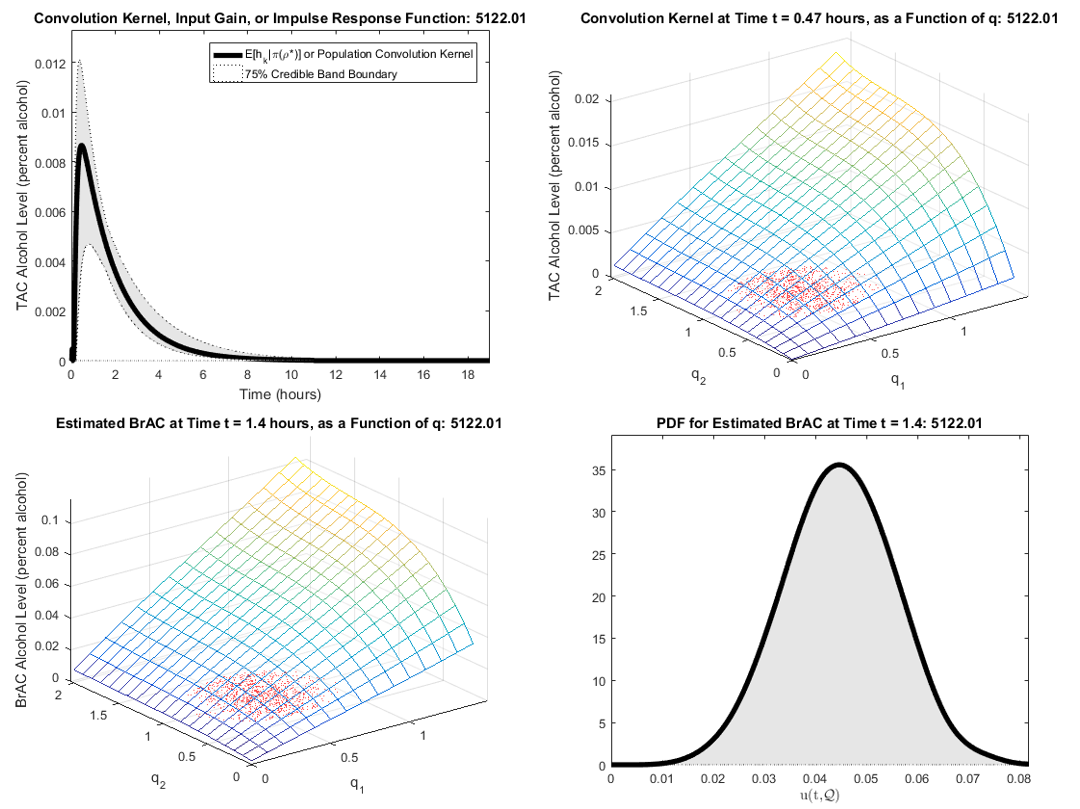}
\caption{Upper Left Panel: Expected value of the impulse response function or convolution kernel together with 75\% credible intervals for population model in Example 7.1. Upper Right Panel: Impulse response function or convolution kernel for population model in Example 7.1 as a function of $q=(q_1,q_2)$. Lower Left Panel: Estimated BrAC, or input signal, $\cu$ at time when the expected value is at its peak as a function of $q=(q_1,q_2)$ for drinking episode 1 in Example 7.1. Lower Right Panel: Probability density function for the estimated BrAC, or input signal, $\cu$ at time when the expected value is at its peak for drinking episode 1 in Example 7.1. The points marked with red dots in the $q_1,q_2$ plane in the upper right and lower left panels are the samples from the bivariate normal distribution that were used to compute the 75\% credible bands.}
\label{fig11}
\end{figure}

\subsection{Example: Multiple Subjects; Data Collected with the Alcohol Monitoring Systems (AMS) SCRAM Alcohol Biosensor}\label{ex2}

\indent

In this example, we used datasets collected from multiple subjects at the University of Illinois at Urbana-Champaign using the AMS SCRAM alcohol biosensor. In this study 60 subjects or participants were given a gender and weight adjusted dose of alcohol (0.82 g/kg for men, 0.74 g/kg for women).  This dosage was selected so as to yield a peak BrAC of approximately .08\%. The alcoholic beverage was administered in 3 equal parts at 0 min, 12 min, and 24 min, and participants were instructed to consume their beverages as evenly as possible over these intervals. All participants were then asked to provide breathalyzer readings at approximately 30 minute intervals. The SCRAM sensor was worn and provided transdermal readings also at a rate of one approximately every 30 minutes. This data was collected for a study other than the one being reported on here.  The design protocol for that study called for each subject’s alcohol challenge session to end when their BrAC dipped below 0.03\% and/or their transdermal readings had reached a peak and begun to descend (see \cite{Fair}). Since typically BrAC leads TAC, in 37 of the original 60 sessions, when the session was halted, the TAC signal had either not yet started its descent, or had not been decreasing for a sufficient amount of time to establish an elimination trend.  For this reason, these sessions were deemed inappropriate for our study here, and were eliminated from further consideration. Of the remaining 23 sessions, 5 displayed what we would characterize as physiologically anomalous behavior in that the TAC led the BrAC. This would be inconsistent with our modeling assumption that there is no alcohol present in the subject’s body at time $t = 0$.  The possible causes for this might include the participant’s failure to adhere to the protocol laid out in the study design, an error in data recording, a sensor hardware malfunction, or the TAC sensor coming into contact with alcohol vapor either as the alcohol dose was being prepared or administered, or from some other source containing ethanol such as skin creams, etc. In any case, we eliminated these 5 subjects from further consideration in our study as well, leaving us with 18 usable participant data sets.
\begin{figure}[t]
\centering
\includegraphics[scale=0.15]{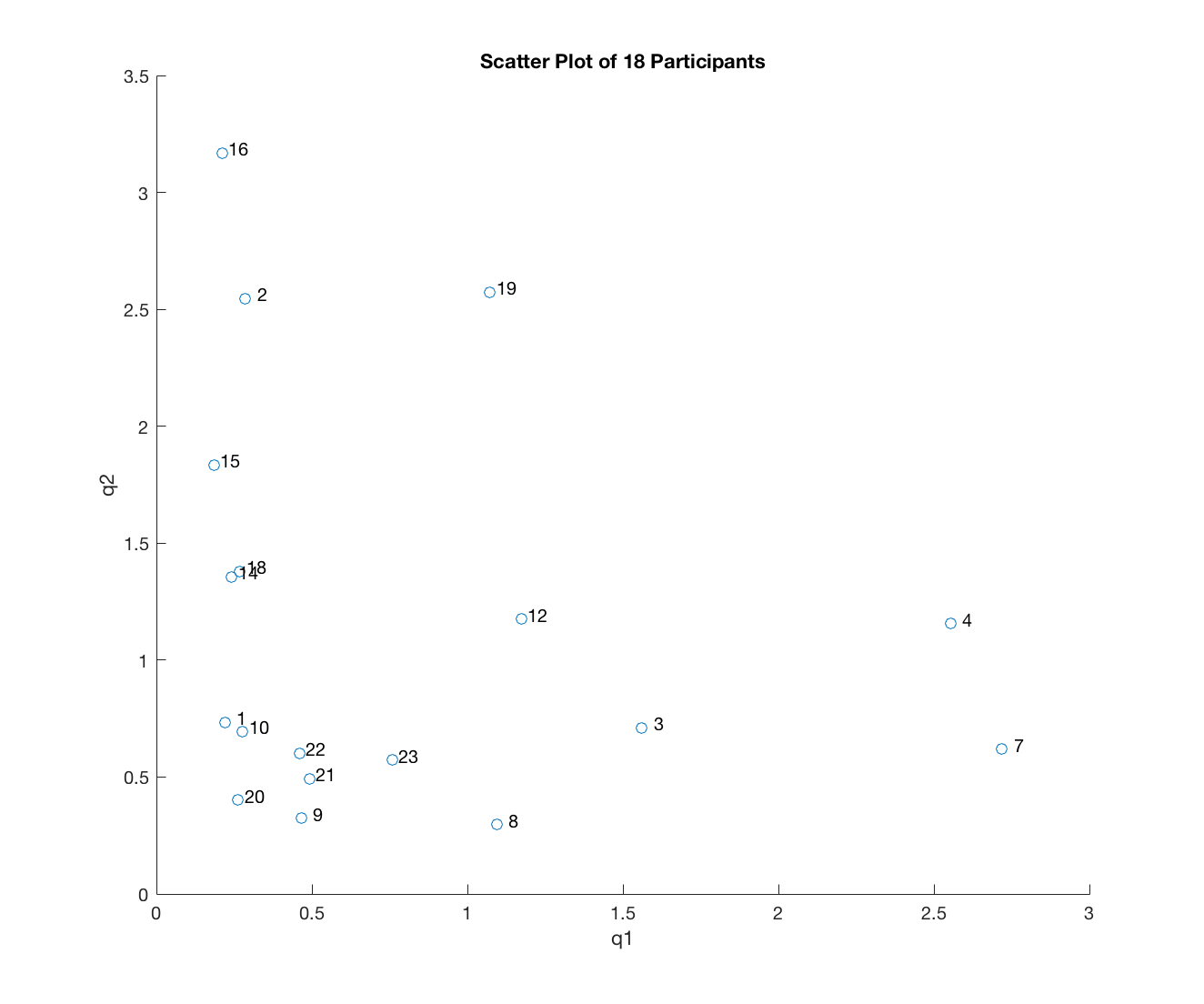}
\caption{Scatter plot of deterministically obtained parameter estimates for the 18 datasets considered in Example 7.2}
\label{fig6}
\end{figure}
\begin{figure}[ht]
\centering
\includegraphics[scale=0.2]{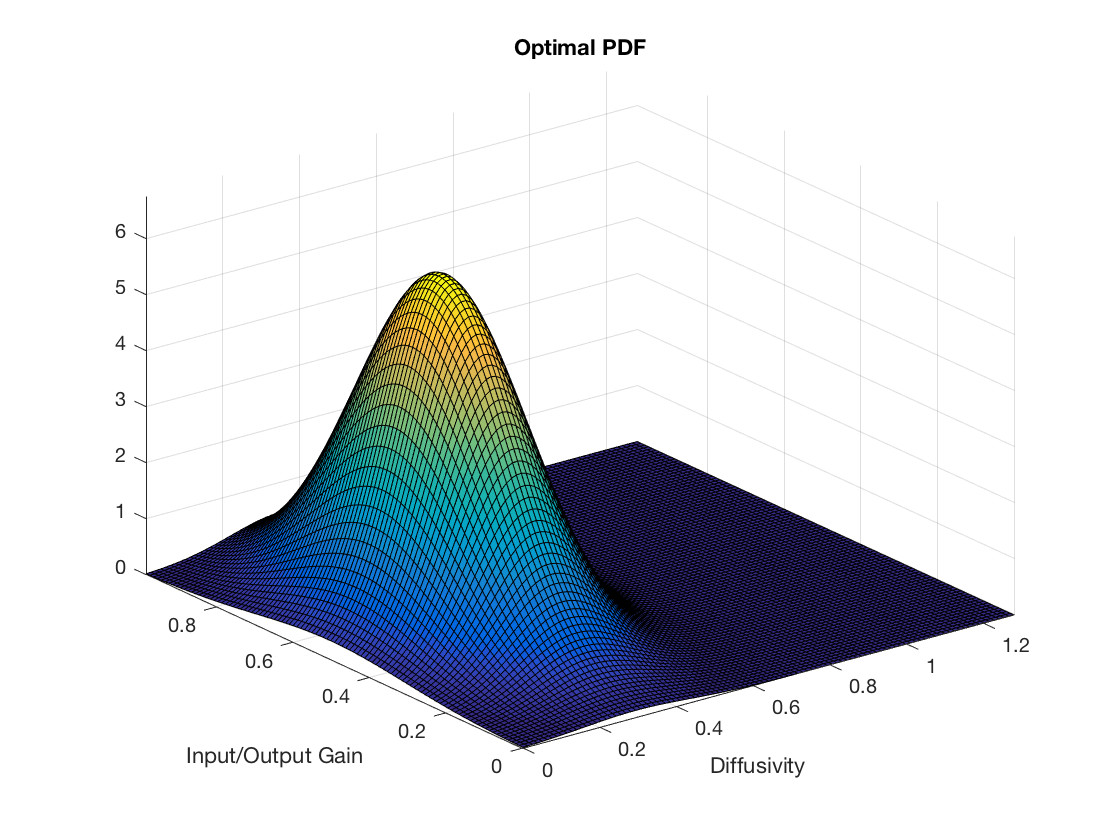}
\caption{Plot of the optimal truncated bivariate normal pdf obtained for Example 7.2.}
\label{fig7}
\end{figure}
We first used our deterministic model to estimate the values for $q=(q_1,q_2)$ for each of the 18 datasets separately. The scatter plot of these $q$ values can be seen in Figure \ref{fig6}. According to this plot, we again visually stratified the datasets by selecting the data of eight of the subjects, subjects 1, 8, 9, 10, 20, 21, 22, and 23, whose $q$ values clustered around the origin. We then used subjects 9 and 21 for cross validation and used the remaining six subjects, numbers 1, 8, 10, 20, 22, and 23, for training.

Following the same steps as we did in Example 7.1 above, for this training population we found the optimal values for the support of the pdf of the truncated bivariate normal distribution to be $a_1^*=0$, $b_1^*=1.2796$ and $a_2^*=0$, $b_2^*=0.9834$. We determined the optimal mean and covariance matrix to be
\begin{align*}
    \mu^*=\begin{bmatrix}
    0.3296\\
    0.3418
    \end{bmatrix}
    \ \ \ \ \text{and} \ \ \ \
    \Sigma^*=\begin{bmatrix}
    0.0187 & 0.0023\\
    0.0023 & 0.0378
    \end{bmatrix}.
\end{align*}
The plot of the optimal pdf is shown in Figure \ref{fig7}. Similarly, we obtained the optimal regularization parameters for this example, as well. For the case in which $\cu$ depends on both time, $t$, and the random parameters, $\cq$, these parameters were found to be $r_1^*=0.0000$ and $r_2^*=3.1877$ while they were $r_1^*=0.0503$ and $r_2^*=5.0974$ for the case $u$ depending only on time, $t$.
\begin{figure}[ht]
\centering
\includegraphics[scale=0.35]{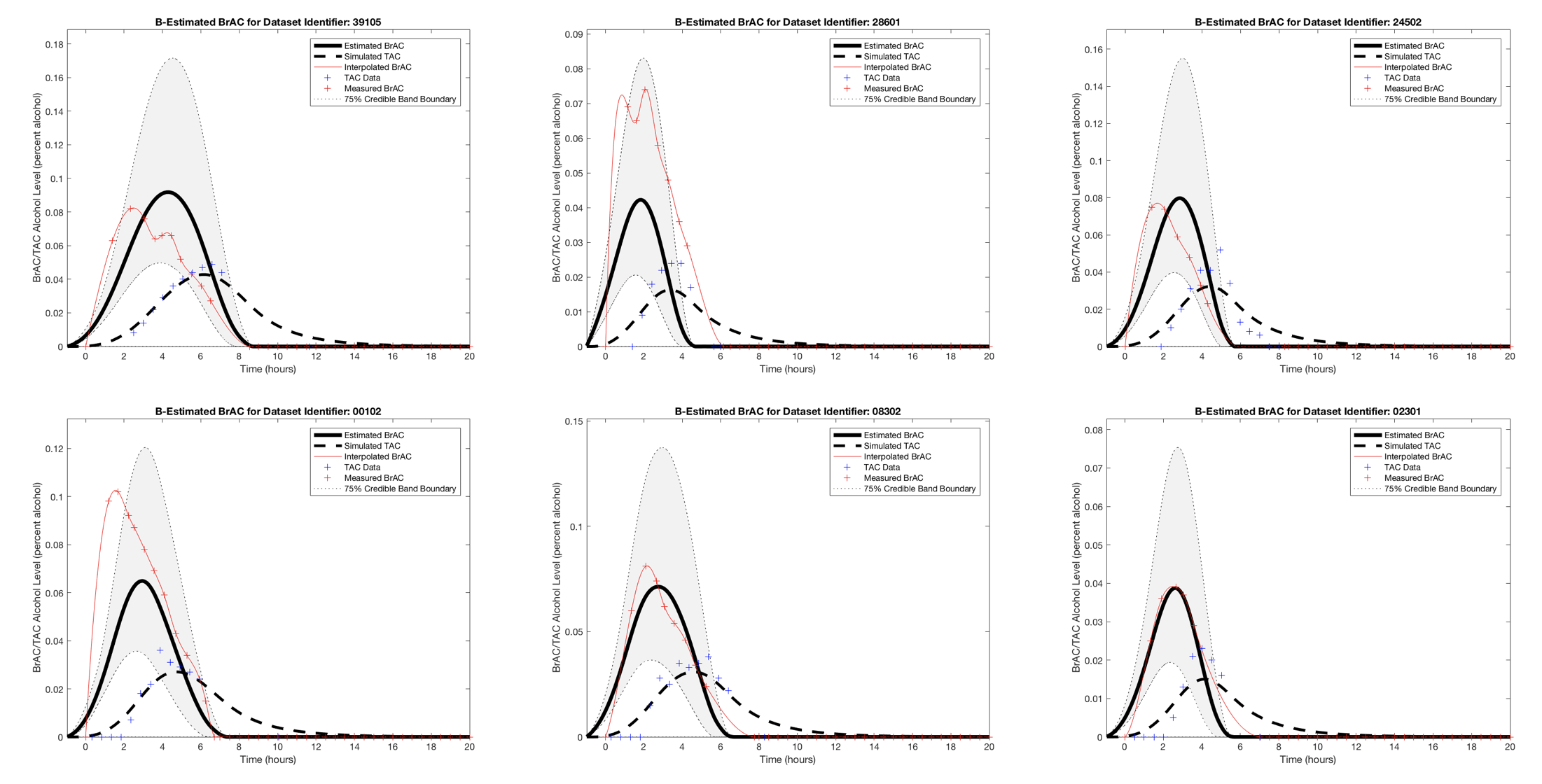}
\caption{Results for Example 7.2: BrAC estimates for the population or training datasets, 1, 8, 10, 20, 22, and 23, assuming the input, $\cu$, depends on both time, $t$, and the random parameters, $\cq$.}
\label{fig8}
\end{figure}

The input or BrAC estimates for the subjects 1, 8, 10, 20, 22, and 23, which were the ones used for training, can be seen in Figures \ref{fig8} and \ref{fig9}. Our results for the case $\cu$ depends on both $t$ and $\cq$ are in Figure \ref{fig8}, and our results for the case $u$ depends on $t$ only are  shown in Figure \ref{fig9}. Analogous results for the cross validation subjects, 9 and 21, can be found in Figure \ref{fig10}.  Figure \ref{fig12} has the plots for Example 7.2 that are analogous to the plots in \ref{fig11} for Example 7.1.
\begin{figure}[ht]
\centering
\includegraphics[scale=0.35]{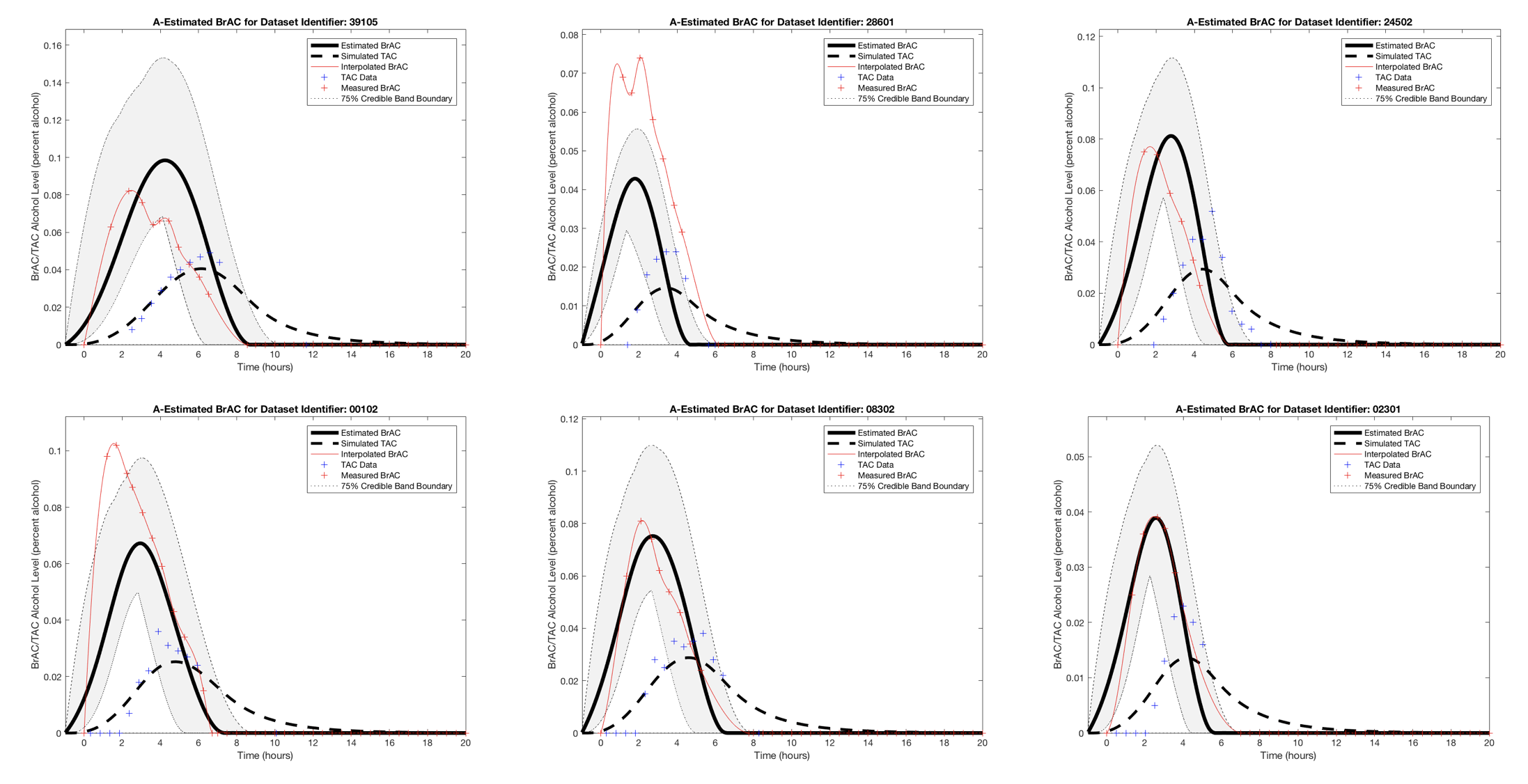}
\caption{Results for Example 7.2: BrAC estimates for the population or training datasets, 1, 8, 10, 20, 22, and 23, assuming the input, $u$, depends only on time, $t$.}
\label{fig9}
\end{figure}
\begin{figure}[ht]
\centering
\includegraphics[scale=0.35]{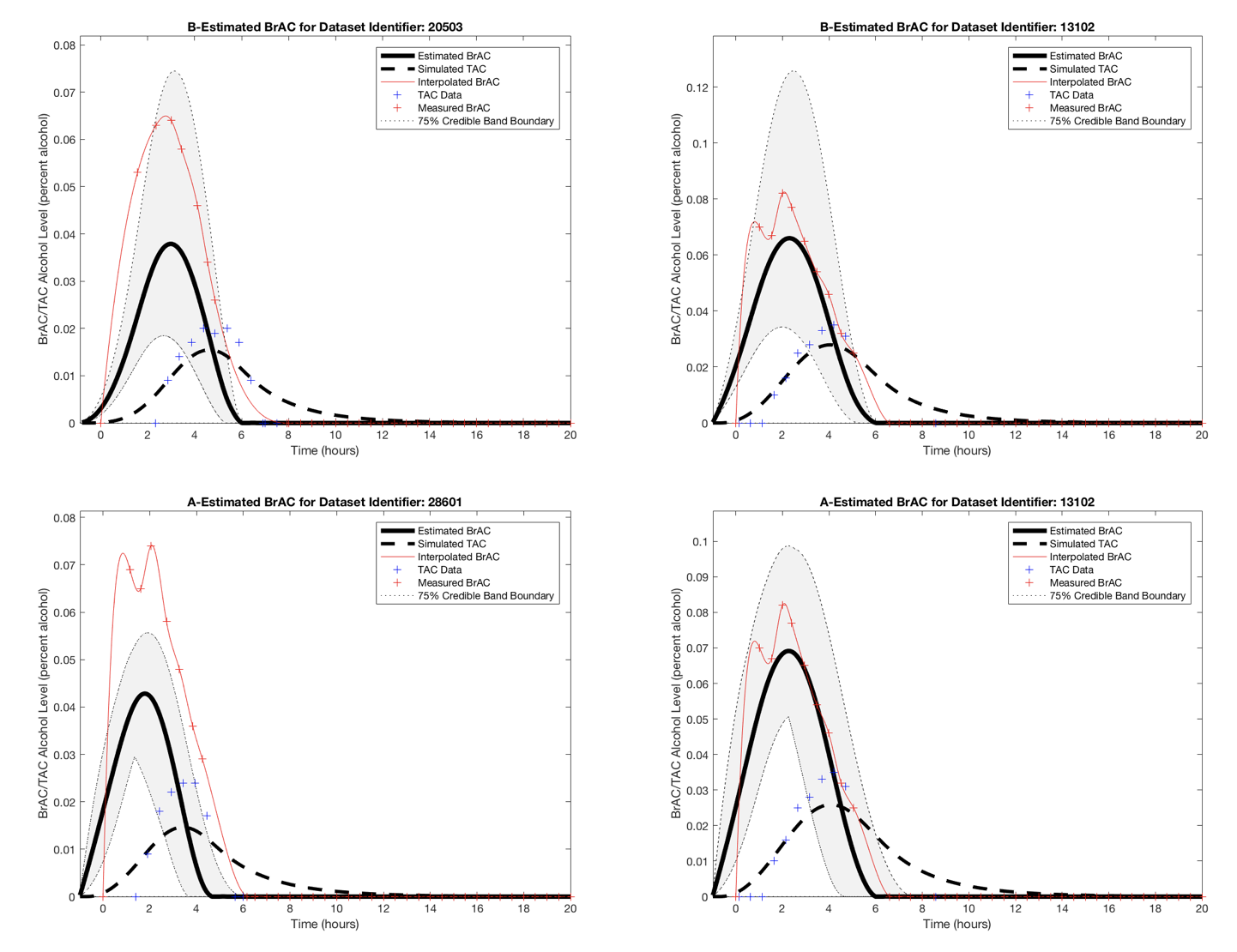}
\caption{Results for Example 7.2: Top row: BrAC estimates for the cross validation datasets, 9 and 21, assuming the input, $\cu$, depends on both time, $t$, and the random parameters, $\cq$. Bottom row: BrAC estimates for the cross validation datasets, 9 and 21, assuming the input, $u$, depends only on time, $t$. }
\label{fig10}
\end{figure}
\begin{figure}[ht]
\centering
\includegraphics[scale=0.5]{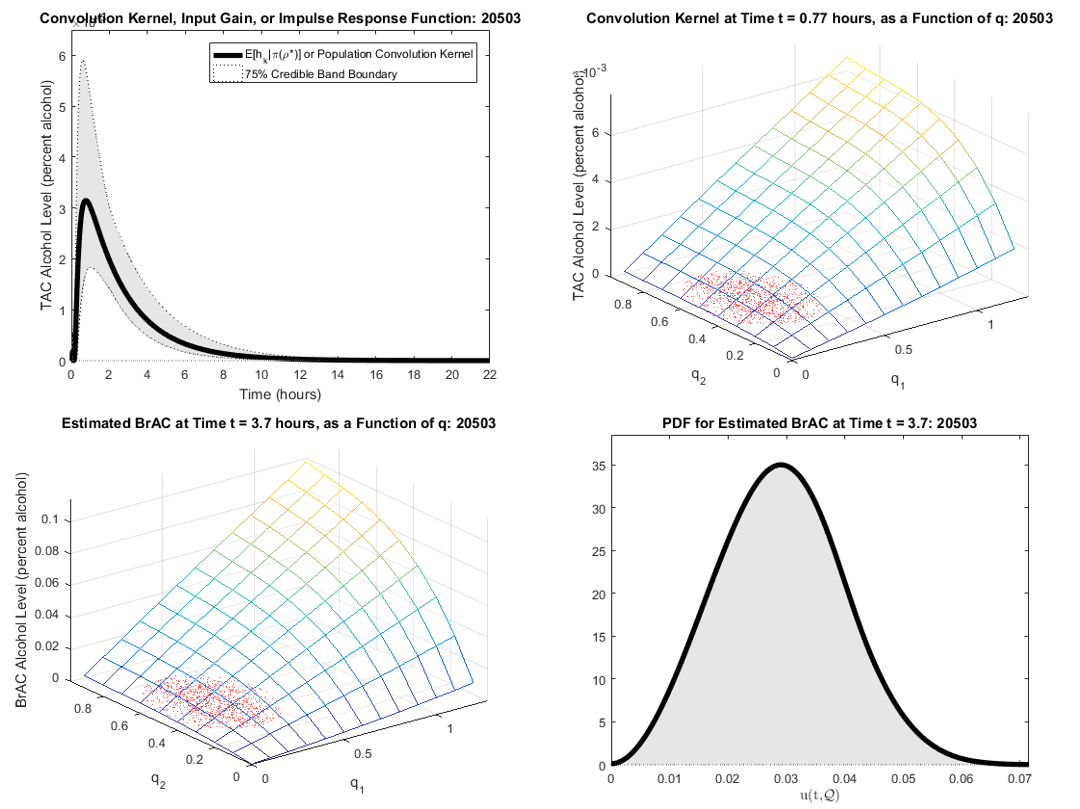}
\caption{Upper Left Panel: Expected value of the impulse response function or convolution kernel together with 75\% credible intervals for population model in Example 7.2. Upper Right Panel: Impulse response function or convolution kernel for population model in Example 7.2 as a function of $q=(q_1,q_2)$. Lower Left Panel: Estimated BrAC, or input signal, $\cu$ at time when the expected value is at its peak as a function of $q=(q_1,q_2)$ for subject number 9 in Example 7.2. Lower Right Panel: Probability density function for the estimated BrAC, or input signal, $\cu$ at time when the expected value is at its peak for subject number 9 in Example 7.2. The points marked with red dots in the $q_1,q_2$ plane in the upper right and lower left panels are the samples from the bivariate normal distribution that were used to compute the 75\% credible bands.}
\label{fig12}
\end{figure}

\newpage
\section{Discussion and Conclusion}\label{disc}

\indent

In the two examples in the previous section, we computed estimated BrAC using two approaches, one in which we allowed the estimated BrAC signal to depend on the random parameters, $\cq_1$ and $\cq_2$, in addition to depending on time, and the other where we sought estimated BrAC as a function of time only. The first method has the benefit of allowing for the efficient computation of credible bands by simply sampling the distribution of the random parameters and then directly substituting the samples into an expression for the estimated BrAC as a function of $q_1$ and $q_2$.  In the second method, credible bands had to be computed via sampling and simulation.  With the first method, training takes longer because of the time involved in solving the optimization problem \eqref{eq6.7} associated with the computation of the optimal regularization parameters. Then the deconvolution of the estimated BrAC signal and computation of the associated credible bands is very quick.  On the other hand, for the second approach, the training is quicker than the first, but because of the need to simulate the deterministic \eqref{eq3.1}-\eqref{eq3.3} model with each of the random samples of $\cq_1$ and $\cq_2$, the deconvolution and credible band determination step is significantly slower than the first approach.  Since the training is off-line and the deconvolution and credible band determination is generally on-line, the first approach is desirable.  In comparing the results for the two methods, we observed some, but generally not significant, differences.
\begin{table}[ht]
\centering
\begin{footnotesize}
\begin{tabular}{| c || c | c | c | c | c || c | c | c | c | c |}
 \hline
 \multicolumn{1}{|c||}{\textbf{Ep.}} & \multicolumn{5}{c||}{\textbf{BrAC}} & \multicolumn{5}{c|}{\textbf{Estimated BrAC}}\\
 \hline
  & \textbf{I} & \textbf{II} & \textbf{III} & \textbf{IV} & \textbf{V} & \textbf{I} & \textbf{II} & \textbf{III} & \textbf{IV} & \textbf{V}\\
 \hline
 \textbf{1} & 0.0520 & 0.7500 & 0.1019 & 0.0173 & 0.0693 & 0.0501 & 0.6333 & 0.1233 & 0.0075 & 0.0319\\
 \textbf{2} & 0.0230 & 0.5000 & 0.0163 & 0.0230 & 0.0460 & 0.0122 & 0.9333 & 0.0160 & 0.0086 & 0.0085\\
 \textbf{4*} & 0.0570 & 2.0000 & 0.1666 & 0.0258 & 0.0285 & 0.0571 & 2.9167 & 0.1653 & 0.0207 & 0.0166\\
 \textbf{6} & 0.0230 & 0.5833 & 0.0267 & 0.0162 & 0.0394 & 0.0174 & 1.2333 & 0.0228 & 0.0136 & 0.0121\\
 \textbf{7} & 0.0180 & 1.5833 & 0.0265 & 0.0196 & 0.0114 & 0.0120 & 1.4667 & 0.0192 & 0.0095 & 0.0068\\
 \textbf{8*} & 0.0170 & 0.8333 & 0.0154 & 0.0204 & 0.0204 & 0.0169 & 0.8167 & 0.0290 & 0.0067 & 0.0116\\
 \textbf{11} & 0.0480 & 0.9167 & 0.1181 & 0.0137 & 0.0524 & 0.0507 & 1.7000 & 0.0950 & 0.0244 & 0.0298\\
 \hline
\end{tabular}
\end{footnotesize}
\caption{Statistics for Example 7.1: I: Maximum value (percent alcohol), II: Time of maximum value (hours), III: Area underneath the curve (percent alcohol $\times$ hours), IV: Elimination rate (percent alcohol per hour), V: Absorption rate (percent alcohol per hour). }
\label{tab1}
\end{table}
In addition to the estimated continuous BrAC signals shown in the plots in the previous section, clinicians and alcohol researchers are interested in a number of statistics associated with the the BrAC for a drinking episode.   Specifically, for each drinking episode, they look at: I: The maximum or peak value of the BrAC, II: The time (since the start of the episode) at which the peak value of the BrAC is attained, III: the area underneath the episode's BrAC curve, IV: the BrAC elimination rate, and V: the BrAC absorption rate.   The BrAC elimination rate is defined to be the peak BrAC value divided by the the amount of time that elapses from the time at which the peak BrAC value was attained until the first zero BrAC measurement (more precisely the time at which the BrAC level first sinks below a predefined threshold), and the BrAC absorption rate is defined to be the peak BrAC value divided by the amount of time that elapses from the last zero BrAC measurement (more precisely the time at which the BrAC level first rises above the predefined threshold) until the time at which the peak BrAC value was attained. In Table \ref{tab1} we provide these statistics for the drinking episodes in Example 7.1 and in Table \ref{tab3} we provide them for each subject's drinking episode in Example 7.2. In Tables \ref{tab2} and \ref{tab4} we use our estimated BrAC surfaces and samples from the bivariate normal distribution corresponding to the parameters $\rho^*$ to compute 75\% credible intervals for each of the statistics.  Note that in all of these tables, the episodes and subjects without an asterisk (*) were used in training the two population models, while those marked with an asterisk were held back for cross validation. All of these results were computed using the first approach (BrAC a function of both time and $\cq_1$ and $\cq_2$).  We computed these statistics and their credible intervals using only the first approach. We did not make these computations using the second approach. 
\begin{table}[ht]
\centering
\begin{footnotesize}
\begin{tabular}{| c || c | c | c | c | c |}
 \hline
 \multicolumn{1}{|c||}{\textbf{Ep.}} & \multicolumn{5}{c|}{\textbf{Credible Intervals}}\\
 \hline
  & \textbf{I} & \textbf{II} & \textbf{III} & \textbf{IV} & \textbf{V}\\
 \hline
 \textbf{1} & [0.0286,0.0661] & [0.4167,0.7000] & [0.0644,0.1687] & [0.0060,0.0105] & [0.0203,0.0444]\\
 \textbf{2} & [0.0044,0.0177] & [0.6833,0.9833] & [0.0050,0.0238] & [0.0045,0.0131] & [0.0039,0.0123]\\
 \textbf{4*} & [0.0143,0.0858] & [2.6000,3.0333] & [0.0277,0.2586] & [0.0083,0.0326] & [0.0052,0.0242]\\
 \textbf{6} & [0.0034,0.0286] & [1.0000,1.3167] & [0.0031,0.0399] & [0.0040,0.0212] & [0.0037,0.0173]\\
 \textbf{7} & [0.0008,0.0208] & [1.2667,1.5333] & [0.0007,0.0349] & [0.0011,0.0156] & [0.0009,0.0105]\\
 \textbf{8*} & [0.0090,0.0230] & [0.6000,0.8833] & [0.0141,0.0405] & [0.0043,0.0094] & [0.0060,0.0154]\\
 \textbf{11} & [0.0164,0.0755] & [1.3667,1.8000] & [0.0238,0.1470] & [0.0103,0.0343] & [0.0122,0.0419]\\
 \hline
\end{tabular}
\end{footnotesize}
\caption{Credible intervals for Example 7.1: I: Maximum value (percent alcohol), II: Time of maximum value (hours), III: Area underneath the curve (percent alcohol $\times$ hours), IV: Elimination rate (percent alcohol per hour), V: Absorption rate (percent alcohol per hour).}
\label{tab2}
\end{table}
For each of the episodes in each example, we computed the percentage of BrAC statistics that fell within the approximating 75\% credible intervals.  For Example 7.1, we obtained the following results: Statistic I: 86\%, II: 14\%, III: 100\%, IV: 29\% and V: 0\%.  For Example 7.2, the following results were obtained: Statistic I: 100\%, II: 38\%, III: 88\%, IV: 88\% and V: 25\%.  The statistics that involved the time of the peak BrAC seemed to pose the most significant challenge for the method.
\begin{table}[ht]
\centering
\begin{footnotesize}
\begin{tabular}{| c || c | c | c | c | c || c | c | c | c | c |}
 \hline
 \multicolumn{1}{|c||}{\textbf{Sub.}} & \multicolumn{5}{c||}{\textbf{BrAC}} & \multicolumn{5}{c|}{\textbf{Estimated BrAC}}\\
 \hline
  & \textbf{I} & \textbf{II} & \textbf{III} & \textbf{IV} & \textbf{V} & \textbf{I} & \textbf{II} & \textbf{III} & \textbf{IV} & \textbf{V}\\
 \hline
 \textbf{1} & 0.0820 & 2.3333 & 0.3822 & 0.0132 & 0.0351 & 0.0918 & 4.2833 & 0.4333 & 0.0206 & 0.0175\\
 \textbf{8} & 0.0740 & 2.0500 & 0.2392 & 0.0180 & 0.0361 & 0.0422 & 1.8333 & 0.1229 & 0.0147 & 0.0149\\
 \textbf{9*} & 0.0640 & 3.0000 & 0.2568 & 0.0130 & 0.0213 & 0.0379 & 2.9833 & 0.1251 & 0.0124 & 0.0100\\
 \textbf{10} & 0.0750 & 0.3833 & 0.2590 & 0.0109 & 0.0542 & 0.0798 & 2.8333 & 0.2559 & 0.0278 & 0.0211\\
 \textbf{20} & 0.01020 & 1.6833 & 0.3820 & 0.0203 & 0.0606 & 0.0649 & 2.9333 & 0.2459 & 0.0146 & 0.0167\\
 \textbf{21*} & 0.0820 & 2.0000 & 0.2884 & 0.0178 & 0.0410 & 0.0660 & 2.3000 & 0.2412 & 0.0177 & 0.0200\\
 \textbf{22} & 0.0810 & 2.1000 & 0.2905 & 0.0143 & 0.0386 & 0.0714 & 2.7500 & 0.2733 & 0.0180 & 0.0191\\
 \textbf{23} & 0.0390 & 2.6500 & 0.1444 & 0.0089 & 0.0147 & 0.0387 & 2.6167 & 0.1160 & 0.0126 & 0.0109\\
 \hline
\end{tabular}
\end{footnotesize}
\caption{Statistics for Example 7.2: I: Maximum value (percent alcohol), II: Time of maximum value (hours), III: Area underneath the curve (percent alcohol $\times$ hours), IV: Elimination rate (percent alcohol per hour), V: Absorption rate (percent alcohol per hour).}
\label{tab3}
\end{table}
Our results suggest a number of open mathematical questions that deserve further consideration.  In our treatment here, although we evaluate them at specific values of $q_1$ and $q_2$, the surfaces we obtain for the estimated impulse response function and input and the associated convergence theory are in fact only $L_2$ making point-wise evaluation undefined.  We are currently looking at the introduction of some form of parabolic regularization \cite{LJL} into the $q$-dependence of the population model.  By doing this, it may become possible to obtain coercivity of the sesquilinear form \eqref{eq3.5} with respect to a stronger norm, and therefore obtain $H^1$-like well-posedness and convergence of the approximations in the $q$ dependence of the state, input, and output.  

We are also looking at eliminating the requirement that the measures, $\pi$, be defined in terms of a parameterized density.  By employing a different version of the Trotter-Kato semigroup approximation theorem (see \cite{BBC}), approximating subspaces with smoother elements, and results for the approximation of measures, we conjecture that we may be able to directly apply the approximation framework developed in \cite{BT} which involves the estimation of the underlying probability measures directly. Another extension in this same spirit would be to attempt to directly estimate the shape of the density rather than simply its parameters.  It seems that, in this case, existing results for the estimation of functional parameters in PDES may be directly applicable.  Also, one might consider a parameterization for the random parameters in the model in terms of their polynomial chaos expansions much as is done in \cite{GAS}.  

Of course, improved performance of the population model could potentially be obtained with higher fidelity models with higher dimensional parameterization; for example these might include the addition of an advection term, non-constant or functional coefficients (probably with finite dimensional parameterization), damped second order hyperbolic models (e.g. the telegraph equation) for diffusion with finite speed of propagation (see, for example, \cite{Okubo}). It would also be interesting to see if an analogous nonlinear theory could be developed for infinite dimensional systems governed by maximal monotone operators (see, for example, \cite{Barbu}).  

Finally, if alcohol biosensors were to be  incorporated into wearable health monitoring technology (e.g. the Fitbit and the Apple Watch), our approach would have to be modified to produce estimated BrAC in real time. This would be a challenge in light of the inherent latency of the human body's metabolism and transdermal secretion via perspiration of ethanol and the limitations of the analog hardware in the current state-of-the-art transdermal alcohol sensors.  We are currently looking at combining the ideas presented here with a \textit{look-ahead} prediction algorithm based on a hidden Markov model.
\begin{table}[t]
\centering
\begin{footnotesize}
\begin{tabular}{| c || c | c | c | c | c |}
 \hline
 \multicolumn{1}{|c||}{\textbf{Sub.}} & \multicolumn{5}{c|}{\textbf{Credible Intervals}}\\
 \hline
  & \textbf{I} & \textbf{II} & \textbf{III} & \textbf{IV} & \textbf{V}\\
 \hline
 \textbf{1} & [0.0497,0.1715] & [3.8833,4.5167] & [0.2235,0.8166] & [0.0121,0.0407] & [0.0103,0.0314]\\
 \textbf{8} & [0.0206,0.0830] & [1.5667,1.9667] & [0.0553,0.2457] & [0.0085,0.0308] & [0.0080,0.0280]\\
 \textbf{9*} & [0.0185,0.0745] & [2.6667,3.1333] & [0.0563,0.2499] & [0.0069,0.0260] & [0.0055,0.0185]\\
 \textbf{10} & [0.0398,0.1552] & [2.5500,2.9833] & [0.1200,0.5024] & [0.0162,0.0578] & [0.0114,0.0392]\\
 \textbf{20} & [0.0357,0.1203] & [2.6333,3.1000] & [0.1267,0.4635] & [0.0096,0.0284] & [0.0099,0.0297]\\
 \textbf{21*} & [0.0343,0.1257] & [1.9833,2.4667] & [0.1163,0.4683] & [0.0102,0.0356] & [0.0115,0.0363]\\
 \textbf{22} & [0.0365,0.1374] & [2.3667,2.9500] & [0.1293,0.5360] & [0.0111,0.0370] & [0.0109,0.0349]\\
 \textbf{23} & [0.0194,0.0753] & [2.3500,2.7500] & [0.0542,0.2283] & [0.0074,0.0258] & [0.0059,0.0203]\\
 \hline
\end{tabular}
\end{footnotesize}
\caption{Credible intervals for Example 7.2: I: Maximum value (percent alcohol), II: Time of maximum value (hours), III: Area underneath the curve (percent alcohol $\times$ hours), IV: Elimination rate (percent alcohol per hour), V: Absorption rate (percent alcohol per hour).}
\label{tab4}
\end{table}

\bibliographystyle{plain}




\end{document}